%% file: asvrcd_arxiv.tex
\documentclass[11pt]{article}

\include{all_commands}
\graphicspath{{./experiments/images/}}

\usepackage[utf8]{inputenc} 
\usepackage[T1]{fontenc}    
\usepackage{hyperref}       
\usepackage{url}            
\usepackage{amsfonts}       
\usepackage{nicefrac}       
\usepackage{microtype}      
\usepackage[margin=1in]{geometry}

\usepackage{algorithm}
\usepackage{algorithmic}

\usepackage{authblk}

\begin{document}

\title{Variance Reduced Coordinate Descent with Acceleration: \\ New Method With a Surprising Application to Finite-Sum Problems}

\author[1]{Filip Hanzely}
\author[1]{Dmitry Kovalev}
\author[1]{Peter Richt\'{a}rik}

\affil[1]{King Abdullah University of Science and Technology, Thuwal, Saudi Arabia}

\maketitle

\begin{abstract}
We propose an accelerated version of stochastic variance reduced coordinate descent -- ASVRCD. As other variance reduced coordinate descent methods such as SEGA or SVRCD, our method can deal with problems that include a non-separable and non-smooth regularizer, while accessing a random block of partial derivatives in each iteration only. However, ASVRCD incorporates Nesterov's momentum, which offers favorable iteration complexity guarantees over both  SEGA and SVRCD. As a by-product of our theory, we show that a variant of Katyusha~\cite{allen2017katyusha} is a specific case of ASVRCD, recovering the optimal oracle complexity for the finite sum objective.
\end{abstract}

\section{Introduction}
In this paper, we aim to solve the regularized optimization problem
\begin{equation}\label{eq:problem}
\compactify \min_{x\in \R^d}   P(x) := f(x) + \psi(x),
\end{equation}
where function $f$ is convex and differentiable, while the regularizer $\psi$ is convex and non-smooth. Furthermore, we assume that the dimensionality $d$ is large. 

The most standard approach to deal with the huge $d$ is to decompose the space, i.e., use coordinate descent, or, more generally, subspace descent methods~\cite{nesterov2012efficiency, wright2015coordinate, kozak2019stochastic}. Those methods are especially popular as they achieve a linear convergence rate on strongly convex problems while enjoying a relatively cheap cost of performing each iteration.

However, coordinate descent methods are only feasible if the regularizer $\psi$ is separable~\cite{richtarik2014iteration}. In contrast, if $\psi$ is not separable, the corresponding stochastic gradient estimator has an inherent (non-zero) variance at the optimum, and thus the linear convergence rate is not achievable. 

This phenomenon is, to some extent, similar when applying Stochastic Gradient Descent (SGD)~\cite{robbins, nemirovski2009robust} on finite sum objective -- the corresponding stochastic gradient estimator has a (non-zero) variance at the optimum, which prevents SGD from converging linearly. Recently, the issue of sublinear convergence of SGD has been resolved using the idea of control variates~\cite{hickernell2005control}, resulting in famous variance reduced methods such as SVRG~\cite{johnson2013accelerating} and SAGA~\cite{defazio2014saga}. 

Motivated by the massive success of variance reduced methods for finite sums, control variates have been proposed to ``fix'' coordinate descent methods to minimize problem~\eqref{eq:problem} with non-separable $\psi$. To best of our knowledge, there are two such algorithms in the literature -- SEGA~\cite{hanzely2018sega} and SVRCD~\cite{hanzely2019one}, which we now quickly describe.\footnote{VRSSD~\cite{kozak2019stochastic} is yet another stochastic subspace descent algorithm aided by control variates; however, it was proposed to minimize $f$ only (i.e., considers $\psi=0$). }

Let $x^k$ be the current iterate of SEGA (or SVRCD) and suppose that the oracle reveals $\nabla_i f(x^k)$ (for $i$ chosen uniformly at random). The simplest unbiased gradient estimator of $\nabla f(x^k)$ can be constructed as $\tilde{g}^k= d\nabla_i f(x^k)e_i$ (where $e_i\in \R^d$ is the $i$-th standard basis vector). The idea behind these methods is to enrich $\tilde{g}^k$ using a control variate $h^k\in \R^d$, resulting in a new (still unbiased) gradient estimator $g^k$:
\[
g^k  = d\nabla_i f(x^k)e_i - dh_i^ke_i + h^k.
\]

\emph{How to choose the sequence of control variates $\{h^k\}$?} Intuitively, we wish for both sequences $\{h^k\}$ and $\{\nabla f(x^k\})$ to have an identical limit point. In such case, we have $\lim_{k \rightarrow \infty }\mathrm{Var}(g^k) =0$, and thus one shall expect faster convergence. There is no unique way of setting $\{h^k\}$ to have the mentioned property satisfied -- this is where SEGA and SVRCD differ. See Algorithm~\ref{alg:SEGAAS} for details.

\begin{algorithm}[h]
    \caption{SEGA and SVRCD}
    \label{alg:SEGAAS}
    \begin{algorithmic}
        \REQUIRE Stepsize $\alpha>0$, starting point $x^0\in\R^d$, probability vector $p$: $p_i\eqdef \Prob{i\in S} $
        \STATE Set $h^0 = 0 \in \R^d$
        \FOR{ $k=0,1,2,\ldots$ }
        \STATE{Sample random  $S\subseteq \{1,2,\dots,d\}$}
        \STATE{$g^k = \sum \limits_{i\in S}  \frac{1}{\pLi}(\nabla_i f(x^k) - h_i^k)\eLi + h^k$}
        \STATE{$x^{k+1} = \prox(x^k - \alpha g^k)$}
        \STATE{$h^{k+1} = \begin{cases} h^{k} +\sum \limits_{i\in S}( \nabla_i f(x^k) - h^{k}_i)\eLi & \text{for SEGA} \\  \begin{cases} \nabla f(x^k) & \text{w.p. } \probx \\ h^k & \text{w.p. } 1-\probx\end{cases}&  \text{for SVRCD} \end{cases} $}
        \ENDFOR
    \end{algorithmic}
\end{algorithm}

In this work, we continue the above  research  along the lines of variance reduced coordinate descent algorithms, with surprising consequences.

\subsection{Contributions\label{sec:contrib}}
Here we list the main contributions of this paper.

$\triangleright$ {\bf Exploiting prox in SEGA/SVRCD.} Assume that the regularizer $\psi$ includes an indicator function of some affine subspace of $\R^d$. We show that both SEGA and SVRCD might exploit this fact, resulting in a faster convergence rate. As a byproduct, we establish the same result in the more general framework from~\cite{hanzely2019one} (presented in the appendix).

$\triangleright$ {\bf Accelerated SVRCD.} We propose an accelerated version of SVRCD - ASVRCD. ASVRCD is the first accelerated variance reduced coordinate descent to minimize objectives with non-separable, proximable regularizer.\footnote{We shall note that an accelerated version of SEGA was already proposed in~\cite{hanzely2018sega} for $\psi = 0$ -- this was rather an impractical result demonstrating that SEGA can match state-of-the art convergence rate of accelerated coordinate descent from~\cite{allen2016even, nesterov2017efficiency, hanzely2018accelerated}. In contrast, our results cover any convex $\psi$.}

$\triangleright$ {\bf SEGA/SVRCD/ASCRVD generalizes SAGA/L-SVRG/L-Katyusha.} We show a surprising link between SEGA and SAGA. In particular, SAGA is a special case of SEGA; and the new rate we obtain for SEGA recovers the tight complexity of SAGA~\cite{qian2019saga, gazagnadou2019optimal}. Similarly, we recover loopless SVRG (L-SVRG)~\cite{hofmann2015variance, kovalev2019don} along with its best-known rate~\cite{hanzely2019one, lsvrgas} using a result for SVRCD. Lastly, as a particular case of ASVRCD, we recover an algorithm which is marginally preferable to loopless Katyusha (L-Katyusha)~\cite{lsvrgas}: while we recover their iteration complexity result, our proof is more straightforward, and at the same time, the stepsize for the proximal operator is smaller.\footnote{This is preferable especially if the proximal operator has to be estimated numerically.}

\subsection{Preliminaries}

As mentioned in Section~\ref{sec:contrib}, the new results we provide i are particularly interesting if the regularizer $\psi$ contains an indicator function of some affine subspace of $\R^d$. 
 
 \begin{assumption}\label{ass:indicator}

 Assume that $\Popt$ is a projection matrix such that 
 \begin{equation}\label{eq:danjadsou}
\psi (x) = \begin{cases}
 \psi'(x) & {\text if}\quad x \in \{ x^0 + \Range{ \Popt} \} \\
\infty & {\text if} \quad  x \not \in \{ x^0 + \Range{ \Popt} \} \end{cases}  \end{equation}
for some convex function $\psi'(x)$. Furthermore, suppose that the proximal operator of $\psi$ is cheap to compute. 
 \end{assumption}

\begin{remark}
If $\psi$ is convex, there is always some $\Popt$ such that~\eqref{eq:danjadsou} holds as one might choose $\Popt=\mI$.
\end{remark}

Next, we require smoothness of the objective, as well as the strong convexity over the affine subspace given by $\Range{\Popt}$.

\begin{assumption}\label{as:smooth_strongly_convex}
Function $f$ is $\mM$-smooth, i.e.,  for all $ x,y\in \R^d$:\footnote{We define $\| x\|^2 \eqdef \langle x,x\rangle$ and $ \|x \|^2_{\mM} \eqdef \langle \mM x, x\rangle$.}
\[
f(x)\leq f(y) + \langle \nabla f(y),x-y\rangle + \frac{1}{2}\| x-y\|^2_{\mM}.
\]
Function $f$ is $\mu$-strongly convex over $ \{ x^0 + \Range{ \Popt} \}$, i.e.,  for all $ x,y\in\{ x^0 + \Range{ \Popt} \}$:
\begin{align}
 f(x)&\geq f(y) + \langle \nabla f(y),x-y\rangle + \frac{\mu}{2}\| x-y\|^2.\label{eq:sc}
\end{align}
\end{assumption}

\begin{remark}
Smoothness with respect to matrix $\mM$ arises naturally in various applications. For example, if $f(x) = f'(\mA x)$, where $f'$ is $L'$-smooth (for scalar $L'>0$), we can derive that $f$ is $\mM = L'\mA^\top \mA$-smooth.
\end{remark}

In order to stress the distinction between the finite sum setup and the setup from the rest of the paper, we are denoting the finite-sum variables that differ from the non-finite sum case in ${\color{red} \text{red}}$. We thus, {\emph{recommend printing this paper in color}}. 

\section{Better rates for SEGA and SVRCD  \label{sec:sega_is2}}
In this section, we show that a specific structure of nonsmooth function $\psi$ might lead to faster convergence of SEGA and SVRCD.

The next lemma is a direct consequence of Assumption~\ref{ass:indicator} -- it shows that proximal operator of $\psi$ is contractive under $\Popt$-norm.
  \begin{lemma}
Let $\{x^k\}_{k\geq0}$ be a sequence of iterates of Algorithm~\ref{alg:SEGAAS} and let $x^*$ be optimal solution of~\eqref{eq:problem}. Then 
\begin{equation}
x^k  \in \{x^0 + \Range{\Popt}\},\;  x^* \in\{x^0 + \Range{\Popt)}\} . \label{eq:q_identity}
\end{equation}
for all $k$. Furthermore, for any $x,y\in \R^{d}$ we have
\begin{equation}\label{eq:prox_cont}
\| \prox(x ) - \prox(y) \|^2\leq \| x-y\|^2_{\Popt}.
\end{equation}

  \end{lemma}

Next, we state the convergence rate of both SEGA and SVRCD under Assumption~\ref{ass:indicator} as Theorem~\ref{thm:sega_as}. We also generalize the main theorem from~\cite{hanzely2019one} (fairly general algorithm which covers SAGA, SVRG, SEGA, SVRCD, and more as a special case; see Section~\ref{sec:analysis2} of the appendix); from which the convergence rate of SEGA/SVRCD follows as a special case. 

\begin{theorem}\label{thm:sega_as}

Let Assumptions~\ref{ass:indicator},~\ref{as:smooth_strongly_convex} hold and denote $p_i\eqdef \Prob{i\in S}$. Consider vector $v = \sum_{i=1}^d \eLi v_i, v_i\geq 0$ such that
\begin{equation} \label{eq:ESO_sega_good}
\mM^{\frac12} \E{  \sum_{i\in S} \frac{1}{\pLi} \eLi \eLi^\top\Popt \sum_{i\in S} \frac{1}{\pLi} \eLi \eLi^\top} \mM^{\frac12}  \preceq \diag(p^{-1}\circ v),
\end{equation}
where $\diag(\cdot)$ is a diagonal operator.\footnote{Returns matrix with the input on the diagonal, zeros everywhere else.}
Then, iteration complexity of SEGA with $\alpha  = \min_i \frac{\pLi}{4v_i+ \mu}$ is $  \max_i\left(\frac{4v_i + \mu}{\pLi \mu} \right)\log\frac1\epsilon$. At the same time, iteration complexity of SVRCD with $\alpha  = \min_i \frac{1}{4v_i\pLi^{-1}+  \mu\probx^{-1}}$ is $  \left(\frac{4\max_i(v_i\pLi^{-1} )+ \mu \probx^{-1}}{ \mu} \right)\log\frac1\epsilon$. 
\end{theorem}

Let us look closer to convergence rate of SVRCD from Theorem~\ref{thm:sega_as}. The optimal vector $v$ is a solution to the following optimization problem
\[
\min_{v\in \R^d}   \;\; \left(\frac{4\max_i\{v_i\pLi^{-1} \}+ \mu \probx^{-1}}{ \mu} \right)\log\frac1\epsilon  \;\; \text{s. t.}   \; \eqref{eq:ESO_sega_good}\; \text{holds}. 
\]
Clearly, there exists a solution of the form $v\propto p$; let us thus choose $v \eqdef \ccL p$ with $\ccL>0$. In this case, to satisfy~\eqref{eq:ESO_sega_good} we must have 
\begin{equation}\label{eq:ccLdef}
\ccL = \lambda_{\max}\left( \mM^{\frac12} \E{  \sum_{i\in S} \frac{1}{\pLi} \eLi \eLi^\top\Popt \sum_{i\in S} \frac{1}{\pLi} \eLi \eLi^\top} \mM^{\frac12} \right)
\end{equation}
and the iteration complexity of SVRCD becomes $\left(\frac{4\ccL+ \mu \probx^{-1}}{ \mu} \right)\log\frac1\epsilon$.\footnote{We decided to not present this, simplified rate in Theorem~\ref{thm:sega_as} for the following two reasons: 1) it would yields a slightly subpotimal rate of SEGA and 2) the connection of to the convergence rate of SAGA from~\cite{qian2019saga} is more direct via~\eqref{eq:ESO_sega_good}.}

\emph{How does $\Popt$ influence the rate?} As mentioned, one can always consider $\Popt=\mI$. In such a case, we recover the convergence rate of SEGA and SVRCD from~\cite{hanzely2019one}. However, the smaller rank of $\Popt$ is, the faster rate is Theorem~\ref{thm:sega_as} providing. To see this, it suffices to realize that if $\ccL$ is increasing in $\Popt$ (in terms of Loewner ordering). 

\begin{example}
Let $\mM= \mI$ and $S=\{i\}$ with probability $d^{-1}$ for all $1\leq i\leq d$. Given that $\Popt=\mI$, it is easy to see that $\ccL = d$. In such case, the iteration complexity of SVRCD is $\left(\frac{4d + \mu \probx^{-1}}{ \mu} \right)\log\frac1\epsilon$. In the other extreme, if $\Popt=\frac1d ee^\top$, we have $\ccL=1$, which yields complexity (of SVRCD) $\left(\frac{4 + \mu \probx^{-1}}{ \mu} \right)\log\frac1\epsilon$. Therefore, given that $\mu = \cO(\probx)$, the low rank of $\Popt$ caused the speedup of order $\Theta(d)$.
\end{example}

We shall also note that the tight rate of SAGA and L-SVRG might be recovered from Theorem~\ref{thm:sega_as} only using a non-trivial $\Popt$ (see Section~\ref{sec:sagasega}), while the original theory of SEGA and SVRCD only yield a suboptimal rate for both SAGA and L-SVRG.


\paragraph{Connection with Subspace SEGA~\cite{hanzely2018sega}.} Assume that function $f$ is of structure $f(x)=h(\mA x)$. As a consequence, we have $\nabla f(x) = \mA^\top \nabla h(\mA x)$ and thus $\nabla f(x)\in \Range{\mA^\top}$. This fact was exploited by Subspace SEGA in order to achieve a faster convergence rate. Our results can mimic Subspace SEGA by setting $\psi$ to be an indicator function of $x^0 + \Range{\mA^\top}$, given that there is no extra non-smooth term in the objective.

\begin{remark}
Throughout all proofs of this section, we have used a weaker conditions than Assumption~\ref{as:smooth_strongly_convex}. In particular, instead of-$\mM$-smoothness, it is sufficient to have\footnote{By $D_f(x,y)$ we denote Bregman distance between $x,y$, i.e., $D_f(x,y)\eqdef f(x)-f(y)-\langle \nabla f(x)$ } 
$
D_{f}(x,x^*) \geq \frac12 \norm{ \nabla f(x)-\nabla f(x^*) }^2_{\mM^{-1}}
$
for all $x\in \R^d$ (Lemma~\ref{lem:smooth_consequence} shows that it is indeed a consequence of $\mM$ smoothness and convexity). At the same time, instead of $\mu$-strong convexity, it is sufficient to have $\mu$-quasi strong convexity, i.e.,  for all $x\in \{x^0 + \Range{\Popt} \}$:
$
f(x^*) \geq f(x) + \langle \nabla f(x), x^*-x\rangle + \frac{\mu}{2}\|x-x^* \|^2. 
$
However, the accelerated method (presented in Section~\ref{sec:acc}) requires the fully general version of Assumption~\ref{as:smooth_strongly_convex}.
\end{remark}

\section{Connection between SEGA (SVRCD) and SAGA (L-SVRG)  \label{sec:sagasega}}

In this section, we show that SAGA and L-SVRG are special cases of SEGA and SVRCD, respectively. At the same time, the previously tightest convergence rate of SAGA~\cite{gazagnadou2019optimal, qian2019saga} and L-SVRG~\cite{hanzely2019one, lsvrgas} follow from Theorem~\ref{thm:sega_as} (convergence rate of SEGA and SVRCD). 

\subsection{Convergence rate of SAGA and L-SVRG \label{sec:saga}}
We quickly state the best-known convergence rate for both SAGA and L-SVRG to minimize the following objective:
\begin{equation}\label{eq:problem_finitesum}
\compactify \min_{\xx\in \R^\dd}  \PP(\xx) \eqdef \underbrace{\frac1n \sum \limits_{j=1}^n \ff_j(\xx)}_{\eqdef \ff(\xx)} + \ppsi(\xx).
\end{equation}

\begin{assumption} \label{as:finitesum}
 Each $\ff_j$ is convex, $\mmM_j$-smooth and $\ff $ is $\mmu$-strongly convex. 
\end{assumption}

Assuming the oracle access to $\nabla \ff_i(\xx^k)$ for $i\in \sS$ (where $\sS$ is a random subset of $\{1,\dots,n \}$), the minibatch SGD~\cite{pmlr-v97-qian19b} uses moves in the direction of the ``plain'' unbiased stochastic gradient $\frac1n \sum \limits_{i\in \sS}  \frac{1}{\ppp_i}\nabla \ff_i(\xx^k) $ (where $\pp_i \eqdef \Prob{i\in \sS} $).

In contrast, variance reduced methods such as SAGA and L-SVRG enrich the ``plain'' unbiased stochastic gradient with control variates: 
\begin{equation}
\ggg^k = \frac1n\sum \limits_{i\in \sS}  \frac{1}{\ppp_i}\left(\nabla \ff_i(\xx^k) - \mJ^k_{:,i} \right)  +  \frac1n \mJ^k \ee.
\end{equation} 
where $\mJ^k \in \R^{\dd\times n}$ is the control matrix and $\ee \in\R^n$ is vector of ones.
The difference between SAGA and L-SVRG lies in the procedure to update $\mJ^k$; SAGA uses the freshest gradient information to replace corresponding columns in $\mJ^k$; i.e.
\begin{equation}\label{eq:saga_update}
\mJ^{k+1}_{:,i} = \begin{cases}
\nabla \ff_i(\xx^k)  & \text{if } i\in \sS \\
\mJ^{k}_{:,i}   & \text{if } i \not \in \sS.
\end{cases}
\end{equation}
On the other hand, L-SVRG sets $\mJ^k$ to the true Jacobian of $f$ upon a successful, unfair coin toss:
 \begin{equation}\label{eq:lsvrg_update}
\mJ^{k+1} = \begin{cases}
\left[ \nabla \ff_1(\xx^k), \dots,  \nabla \ff_n(\xx^k)\right] & \text{w. p. } \probx \\
\mJ^{k}   & \text{w. p. } 1- \probx.
\end{cases}
\end{equation}

The formal statement of SAGA and L-SVRG is provided in the appendix as Algorithm~\ref{alg:saga}, while Proposition~\ref{prop:sagarate} states their convergence rate.

\begin{proposition}\label{prop:sagarate}
Suppose that Assumption~\ref{as:finitesum} holds and let $\vv$ be a nonegative vector such that for all $h_1,\dots,h_n \in \R^{\dd}$ we have
\begin{equation} \label{eq:ESO_saga}
\E{\left\|\sum_{j \in \sS} \mmM^{\frac12}_{j} h_{j}\right\|^{2}} \leq \sum_{j=1}^{n} \pp_j \vv_{j}\left\|h_{j}\right\|^{2} .
\end{equation}
Then the iteration complexity of SAGA with $ \aalpha  = \min_j \frac{n \pp_j }{4\vv_j + n\mmu}$ is $\max_j \left(  \frac{4\vv_j  + n \mmu }{n  \mmu \pp_j }\right) \log\frac1\epsilon$. At the same time, iteration complexity of L-SVRG with $ \aalpha  = \min_j \frac{n}{4 \frac{\vv_j}{ \pp_j } + \frac{  \mmu n}{\probx}}$is $\max_j \left(  4 \frac{\vv_j}{n  \mmu \pp_j  }  + \frac{1}{ \probx } \right) \log\frac1\epsilon$.
\end{proposition}

\subsection{SAGA is a special case of SEGA}
Consider setup from Section~\ref{sec:saga}; i.e.,  problem~\eqref{eq:problem_finitesum} along with Assumption~\ref{as:finitesum} and $\vv$ defined according to~\eqref{eq:ESO_saga}. We will construct an instance of~\eqref{eq:problem} (i.e.,  specific $f$, $\psi$), which is equivalent to~\eqref{eq:problem_finitesum}, such that applying SEGA on~\eqref{eq:problem} is equivalent applying SAGA on~\eqref{eq:problem_finitesum}.

Let $d \eqdef \dd n$. For convenience, define $R_j \eqdef \{ \dd(j-1)+1,  \dd(j-1)+1, \dots,  \dd j  \}$ (i.e., $|R_j|= \dd$) and lifting operator $\Lift{\cdot}: \R^\dd \rightarrow \R^{d}$ defined as 
$\Lift{\xx} \eqdef \left [\underbrace{\xx^\top, \dots ,\xx^\top}_{n\, \mathrm{times}}\right]^\top$.

\paragraph{Construction of $f$, $\psi$.}  Let $\Ind{}$ be indicator function of the set\footnote{Indicator function of a set returns 0 for each point inside of the set and $\infty$ for each point outside of the set.} $x_{R_1}=\dots=x_{R_n}$ and choose
\begin{eqnarray}
f(x) \eqdef  \frac1n\sum_{j=1}^n \ff_j(x_{R_j}),  \;\; 
 \psi(x) \eqdef  \Ind{} (x) + \ppsi(x_{R_1})\label{eq:equivalent}
\end{eqnarray}
Now, it is easy to see that problem~\eqref{eq:problem_finitesum} and problem~\eqref{eq:problem} with the choice~\eqref{eq:equivalent} are equivalent; each $x\in \R^d$ such that $P(x)<\infty$ must be of the form $x = \Lift{\xx}$ for some $\xx\in \R^\dd$. In such case, we have $P(x) = \PP(\xx)$. The next lemma goes further, and derives the values $\mM, \mu, \Popt$ and $v$ based on $\mmM_i$ ($1\leq i\leq n$), $\mmu, \vv$.

\begin{lemma}\label{lem:equivalent_objectives}
Consider $f, \psi$ defined by~\eqref{eq:equivalent}. Function $f$ satisfies Assumption~\ref{as:smooth_strongly_convex} with $\mu \eqdef \frac{ \mmu}{n}$ and 
$\mM \eqdef \frac1n \blockdiag(\mmM_{1}, \dots, \mmM_n)$. Function $\psi$ and $x^0 = \Lift{\xx^0}$ satisfy Assumption with $\Popt \eqdef \frac1n \ee\ee^\top \otimes \mI$. At the same time, given that $\vv$ satisfies~\eqref{eq:ESO_saga}, inequaility~\eqref{eq:ESO_sega_good} holds with $v = \vv n^{-1}$. 
\end{lemma} 

Next, we show that running Algorithm~\ref{alg:SEGAAS} in this particular setup is equivalent to running Algorithm~\ref{alg:saga} for the finite sum objective.

\begin{lemma}\label{lem:saga_from_sega}
Consider $f, \psi$ from~\eqref{eq:equivalent}, $S$ as described in the last paragraph and 
$x^0 =\Lift{\xx^0}$. Running SEGA (SVRCD) on~\eqref{eq:problem} with $S \eqdef \cup_{j\in \sS} R_j$ and $\alpha \eqdef n\aalpha$ is equivalent to running SAGA (L-SVRG) on~\eqref{eq:problem_finitesum}.; i.e.,  we have for all $k$  
\begin{equation}x^k = \Lift{\xx^k}.\label{eq:iterates_equivalence}\end{equation}
\end{lemma} 

As a consequence of Lemmas~\ref{lem:equivalent_objectives} and~\ref{lem:saga_from_sega}, we get the next result.

\begin{corollary}\label{cor:saga_as2++}  Let $f, \psi, S$ be as described above. Convergence rate of SAGA (L-SVRG) given by Proposition~\ref{prop:sagarate} to solve~\eqref{eq:problem} is identical to convergence rate of SEGA (SVRCD) given by Theorem~\ref{thm:sega_as}.
\end{corollary}

\section{Accelerated SVRCD \label{sec:acc}}
In this section we present SVRCD with Nesterov's momentum~\cite{nesterov1983method} -- ASVRCD. 
The development of ASVRCD along with the theory (Theorem~\ref{thm:acc}) was motivated by Katyusha~\cite{allen2017katyusha}, ASVRG~\cite{asvrg} and their loopless variants~\cite{kovalev2019don, lsvrgas}. In Section~\ref{sec:kat_special}, we show that a variant of L-Katyusha (Algorithm~\ref{alg:katyusha2}) is a special case of ASVRCD, and argue that it is slightly superior to the methods mentioned above.

The main component of ASVRCD is the gradient estimator $g^k$ constructed analogously to SVRCD. In particular, $g^k$ is an unbiased estimator of $\nabla f(x^k)$ controlled by $\nabla f(w^k)$:\footnote{This is efficient to implement as sequence of iterates $\{w^k\}$ is updated rarely.}
\begin{equation}
g^k = \nabla f(w^k)+\sum \limits_{i\in S}  \frac{1}{\pLi}(\nabla_i f(x^k) - \nabla_i f(w^k))\ones_i.
\end{equation}

Next, ASVRCD requires two more sequences of iterates $\{ y^k\}_{k\geq0}, \{ z^k\}_{k\geq 0}$ in order to incorporate Nesterov's momentum. The update rules of those sequences consist of subtracting $g^k$ alongside with convex combinations or interpolations of the iterates. See Algorithm~\ref{alg:acc} for specific formulas.  

\begin{algorithm}[h]
	\caption{Accelerated SVRCD (ASVRCD)}
	\label{alg:acc}
	\begin{algorithmic}
		\REQUIRE $0< \theta_1, \theta_2 <1$, $\eta, \beta , \gamma > 0$, $\probx \in (0,1)$, $y^0 = z^0 = x^0 \in \R^d$
		\FOR{$k=0,1,2,\ldots$}
			\STATE $x^k = \theta_1 z^k + \theta_2 w^k + ( 1 -\theta_1 -\theta_2) y^k$
        \STATE{Sample random  $S\subseteq \{1,2,\dots,d\}$}
        \STATE{$g^k = \nabla f(w^k)+\sum \limits_{i\in S}  \frac{1}{\pLi}(\nabla_i f(x^k) - \nabla_i f(w^k))\ones_i$}
			\STATE $y^{k+1} = \proxop_{\eta \psi} (x^k - \eta g^k)$
			\STATE $z^{k+1} = \beta z^k + (1-\beta)x^k + \frac{\gamma}{\eta}(y^{k+1} - x^k)$
			\STATE $w^{k+1} = \begin{cases}
				y^k, &\text{ with probability } \probx\\
				w^k, &\text{ with probability } 1-\probx\\
			\end{cases}$
		\ENDFOR
	\end{algorithmic}
\end{algorithm}

We are now ready to present ASVRCD along with its convergence guarantees.

\begin{theorem} \label{thm:acc}
Let Assumption~\ref{ass:indicator},~\ref{as:smooth_strongly_convex} hold and denote $L \eqdef \lambda_{\max} \left( \mM^{\frac12} \Popt \mM^{\frac12}\right)$.
Further, let $\cL$ be such that for all $k$ we have
	\begin{equation}\label{eq:exp:smooth}
		\E{\norm{g^k - \nabla f(x^k)}^2_\Popt} \leq 2\cL D_f(w^k,x^k).
	\end{equation}
 Define the following Lyapunov function:
	\begin{eqnarray*}
	\Psi^k &\eqdef&  \norm{z^k - x^*}^2 + \frac{2\gamma\beta}{\theta_1}\left[P(y^k) - P(x^*)\right] \\
	&&\qquad+ \frac{(2\theta_2 + \theta_1)\gamma\beta}{\theta_1\probx}\left[P(w^k) - P(x^*)\right],
	\end{eqnarray*}
and let
\begin{eqnarray*}
\eta &=&  \frac14 \max\{\cL, L\}^{-1}, \\
\theta_2 &=& \frac{\cL}{2\max\{L, \cL\}}, \\
\gamma &=& \frac{1}{\max\{2\mu, 4\theta_1/\eta\}},\\
 \beta &=& 1 - \gamma\mu \; \mathrm{and} \\
 \theta_1 &=& \min\left\{\frac{1}{2},\sqrt{\eta\mu \max\left\{\frac{1}{2}, \frac{\theta_2}{\rho}\right\}}\right\} .
\end{eqnarray*}
	Then the following inequality holds:
	\begin{equation*}
		\E{\Psi^{k+1}} \leq
		\left[1 -  \frac{1}{4}\min\left\{\probx, \sqrt{\frac{\mu}{2\max\left\{L, \frac{\cL}{\rho}\right\}}} \right\} \right]\Psi^0.
	\end{equation*}

	As a consequence, iteration complexity of Algorithm~\ref{alg:acc} is 
	$\cO\left( \left( \frac{1}{\probx} +  \sqrt{\frac{L}{\mu}}  + \sqrt{\frac{\cL}{\probx\mu}}  \right)\log\frac1\epsilon\right)$.
\end{theorem}

Convergence rate of ASVRCD depends on constant $\cL$ such that~\eqref{eq:exp:smooth} holds. The next lemma shows that $\cL$ can be obtained indirectly from $\mM$-smoothness (via $\ccL$), in which case the convergence rate provided by Theorem~\ref{thm:acc} significantly simplifies. 

\begin{lemma}\label{lem:exp_smooth_vs_ESO}
Inequality~\ref{eq:exp:smooth} holds for $\cL = \ccL$ (defined in~\eqref{eq:ccLdef}). Further, we have $L\leq \ccL$. Therefore, setting $ \probx \geq \sqrt{\frac{\mu}{\ccL}}$ yields the following complexity of ASVRCD: 
\begin{equation}\label{eq:rate_simple}
\cO\left( \sqrt{\frac{\ccL}{\probx\mu}} \log\frac1\epsilon\right).
\end{equation}
\end{lemma}

Setting $\cL = \ccL$ might be, however, loose in some cases. In particular, inequality~\eqref{eq:exp:smooth} is slightly weaker than~\eqref{eq:ESO_sega_good} and consequently, the bound bound from Theorem~\ref{thm:acc} is slightly better than~\eqref{eq:rate_simple}. To see this, notice that the proof of Lemma~\ref{lem:exp_smooth_vs_ESO} bounds variance of $g^k + \nabla f(w^k)$ by its second moment. Admittedly, this bound might not worsen the rate by more than a constant factor when $\frac{\E{|S|}}{d}$ is not close to 1. Therefore, bound~\eqref{eq:rate_simple} is good in essentially all practical cases. The next reason why we keep inequality~\eqref{eq:exp:smooth} is that an analogous assumption was required for the analysis of L-Katyusha in~\cite{lsvrgas} (see Section~\ref{sec:katyusha}) -- and so we can now recover L-Katyusha results directly.

Let us give a quick taste how the rate of ASVRCD behaves depending on $\Popt$. In particular, Lemma~\ref{lem:acc_example} shows that nontrivial $\Popt$ might lead to speedup of order $\Theta(\sqrt{d})$ for ASVRCD. 

\begin{lemma} \label{lem:acc_example}
Let $S = i$ for each $1\leq i\leq d$ with probability $\frac1d$ and $\probx = \frac1d$. Then, if $\Popt = \mI$, iteration complexity of ASVRCD is $\cO\left(d\sqrt{\frac{\lambda_{\max} \mM}{\mu}} \log \frac1\epsilon\right)$. If, however, $\Popt = \frac1d ee^\top$, iteration complexity of ASVRCD is $\cO\left(\sqrt{\frac{d\lambda_{\max} \mM}{\mu}} \log \frac1\epsilon\right)$.
\end{lemma}

\section{Connection between ASVRCD and L-Katyusha}
Next, we show that L-Katyusha can be seen as a particular case of ASVRCD. 

\subsection{Convergence rate of L-Katyusha\label{sec:katyusha}}
In this section, we quickly introduce the loopless Katyusha (L-Katyusha) from~\cite{lsvrgas} along with its convergence guarantees. In the next section, we show that an improved version of L-Katyusha can be seen as a special case of ASVRCD, and at the same time, the tight convergence guarantees from~\cite{lsvrgas} can be obtained as a special case of Theorem~\ref{thm:acc}.

Consider problem~\eqref{eq:problem_finitesum} and suppose that  $\ff $ is $\LL$-smooth and $\mmu$-strongly convex. Let $\sS$ be a random subset of $\{1, \dots ,n \}$ (sampled from arbitrary fixed distribution) such that $\ppp_i \eqdef \Prob{i\in \sS}$. For each $k$ let $\ggg^k$ be the following unbiased, variance reduced estimator of $\nabla \ff(x^k)$: 

\[
\ggg^k = \frac1n \left( \sum_{i\in \sS} \ppp_i^{-1} \left(\nabla \ff_i(\xx^k) - \nabla \ff_i(\ww^k) \right) \right) + \nabla \ff(\ww^k).
\]

Next, L-Katyusha requires the variance of $\ggg^k$ to be bounded by Bregman distance between $\ww^k$ and $\xx^k$ with constant $\cLL$, as the next assumption states.

\begin{assumption}\label{as:Katyusha} For all $k$ we have
\begin{equation}
\E{\| \ggg^k - \nabla\ff(\xx^k) \|^2} \leq 2 \cLL D_f(\ww^k,\xx^k).
\end{equation}
\end{assumption}
 Proposition~\ref{prop:katyusha} provides a convergence rate of L-Katyusha. 
\begin{proposition}\cite{lsvrgas} \label{prop:katyusha} Let $\ff $ be $\LL$-smooth and $\mmu$-strongly convex while Assumption~\ref{as:Katyusha} holds. 
Iteration complexity of L-Katyusha is $\cO\left( \left(\frac{1}{\pp} + \sqrt{\frac{\LL}{\mmu}} + \sqrt{\frac{\cLL}{\mmu \pp}} \right)\log \frac1\epsilon \right)$.
\end{proposition}

\subsection{L-Katyusha is a special case of ASVRCD\label{sec:kat_special}}
In this section, we show that a modified version of L-Katyusha (Algorithm~\ref{alg:katyusha2}) is a special case of ASVRCD. Furthermore, we show that the tight convergence rate of L-Katyusha~\cite{lsvrgas} follows from Theorem~\ref{thm:acc} (convergence rate of ASVRCD).

Consider again $f,\psi$ chosen according to~\eqref{eq:equivalent}. With this choice, problem~\eqref{eq:problem} and~\eqref{eq:problem_finitesum} are equivalent. At the same time, Lemma~\ref{lem:saga_from_sega} establishes that $f$ satisfies Assumption~\ref{as:smooth_strongly_convex} with $\mu =\frac{ \mmu}{n}$ and $\mM = \frac1n \blockdiag(\mmM_{1}, \dots, \mmM_n)$ while $\psi$ and $x^0$ satisfy Assumption with $\Popt \eqdef \frac1n \ee\ee^\top \otimes \mI$.

 Note that the update rule of sequences $x^k, z^k, w^k$ are identical for both algorithms; we shall thus verify that the update rule on $y^k$ is identical as well. The last remaining thing is to relate $\cL$ and $\cLL$. The next lemma establishes both results.

\begin{lemma}\label{lem:katyusha_from_asvrcd}
Running ASVRCD on~\eqref{eq:problem} with $S \eqdef \cup_{j\in \sS} R_j$ and $\eta \eqdef n\eeta$, $\gamma \eqdef n\ggamma$ is equivalent to running Algorithm~\ref{alg:katyusha2} on~\eqref{eq:problem_finitesum}. At the same time, inequality~\ref{eq:exp:smooth} holds with $\cL = n^{-1}\cLL$, while we have $L = n^{-1}\LL$.
\end{lemma} 

As a direct consequence of Lemma~\ref{lem:katyusha_from_asvrcd} and Theorem~\ref{thm:acc}, we obtain the next corollary.

\begin{corollary}\label{cor:saga_as2++}  Let $f, \psi, S$ be as described above. Iteration complexity of Algorithm~\ref{alg:katyusha2} is \[\cO\left( \left(\frac{1}{\pp} + \sqrt{\frac{\LL}{\mmu}} + \sqrt{\frac{\cLL}{\mmu \pp}} \right)\log \frac1\epsilon \right).\]
\end{corollary}

As promised, the convergence rate of Algorithm~\ref{alg:katyusha2} matches the convergence rate of L-Katyusha from Proposition~\ref{prop:katyusha} and thus matches the lower bound for finite sum minimization by~\cite{woodworth2016tight}. Let us now argue that Algorithm~\ref{alg:katyusha2} is slightly superior to other accelerated SVRG variants. 

First, Algorithm~\ref{alg:katyusha2} is loopless; thus has a simpler analysis and slightly better properties (as shown by~\cite{kovalev2019don}) over Katyusha~\cite{allen2017katyusha} and ASVRG~\cite{asvrg}. Next, the analysis is simpler than~\cite{lsvrgas} (i.e., we do not require one page of going through special cases). At the same time, Algorithm~\ref{alg:katyusha2} uses a smaller stepsize for the proximal operator than L-Katyusha, which is useful if the proximal operator does is estimated numerically. However, Algorithm~\ref{alg:katyusha2} is almost indistinguishable from L-Katyusha if $\ppsi=0$.

\begin{remark}
The convergence rate of L-Katyusha from~\cite{lsvrgas} allows exploiting the strong convexity of regularizer $\psi$ (given that it is strongly convex). While such a result is possible to obtain in our case, we have omitted it for simplicity. 
\end{remark}

\begin{algorithm}[h]
	\caption{Variant of L-Katyusha (special case of Algorithm~\ref{alg:acc})}
	\label{alg:katyusha2}
	\begin{algorithmic}
		\REQUIRE $0< \theta_1, \theta_2 <1$, $\eeta, \beta , \ggamma > 0$, $\probx \in (0,1)$
		\STATE $\yy^0 = \zz^0 = \xx^0 \in \R^\dd$
		\FOR{$k=0,1,2,\ldots$}
			\STATE $\xx^k = \theta_1 \zz^k + \theta_2 \ww^k + ( 1 -\theta_1 -\theta_2) \yy^k$
        \STATE{Sample random  $\sS\subseteq \{1,2,\dots,n\}$}
        \STATE{$g^k = \nabla \ff(\ww^k)+\sum \limits_{i\in \sS}  \frac{1}{\ppp_i}(\nabla \ff_i(\xx^k) - \nabla \ff_i(\ww^k)) $}
			\STATE $\yy^{k+1} = \proxop_{\eeta \psi} (\xx^k - \eeta \ggg^k)$
			\STATE $\zz^{k+1} = \beta \zz^k + (1-\beta)\xx^k + \frac{\ggamma}{\eeta}(\yy^{k+1} - \xx^k)$
			\STATE $\ww^{k+1} = \begin{cases}
				\yy^k, &\text{ with probability } \probx\\
				\ww^k, &\text{ with probability } 1-\probx\\
			\end{cases}$
		\ENDFOR
	\end{algorithmic}
\end{algorithm}

\section{Experiments \label{sec:experiments}}

In this section, we numerically verify the performance of ASVRCD, as well as the improved performance of SVRCD under Assumption~\ref{ass:indicator}. In order to better understand and control the experimental setup, we consider a quadratic minimization (four different types) over the unit ball intersected with a linear subspace.\footnote{Note that the practicality of ASVRCD immediately follows as it recovers Algorithm~\ref{alg:katyusha2} as a special case, which is (especially for $\psi=0$) almost indistinguishable to L-Katyusha -- state-of-the-art method for smooth finite sum minimization. For this reason, we decided to focus on less practical, but better-understood experiments.} The specific choice of the objective is presented in Section~\ref{sec:exp_choice} of the Appendix.

In the first experiment we demonstrate the superiority of ASVRCD to SVRCD for problems with $\Popt = \mI$. We consider four different methods -- ASVRCD and SVRCD, both with uniform and importance sampling such that $|S|=1$ with probability 1. The importance sampling is the same as one from~\cite{hanzely2019one}. In short, the goal is to have $\ccL$ from~\eqref{eq:ccLdef} as small as possible. Using $\Popt = \mI$, it is easy to see that $\ccL = \lambda_{\max} \left( \diag(p)^{-\frac12}\mM  \diag(p)^{-\frac12}\right)$. While the optimal $p$ is still hard to find, we set $p_i\propto \mM_{i,i}$ (i.e., the effect of importance sampling is the same as the effect of Jacobi preconditioner). Figure~\ref{fig:identity} shows the result. As expected, accelerated SVRCD always outperforms non-accelerated variant, while at the same time, the importance sampling improves the performance too.

\begin{figure}[!h]
\centering
\begin{minipage}{0.33\textwidth}
  \centering
\includegraphics[width =  \textwidth ]{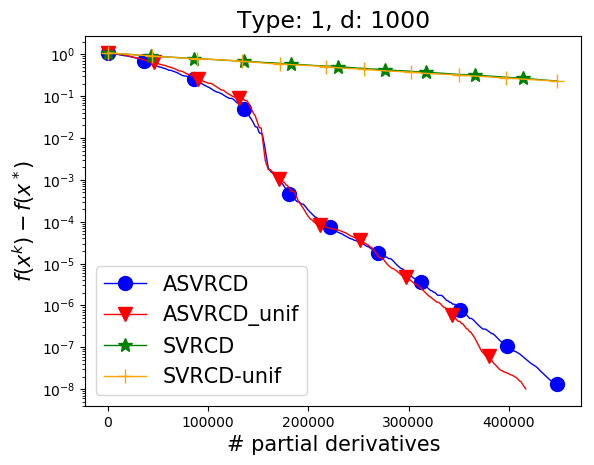}
\end{minipage}%
\begin{minipage}{0.33\textwidth}
  \centering
\includegraphics[width =  \textwidth ]{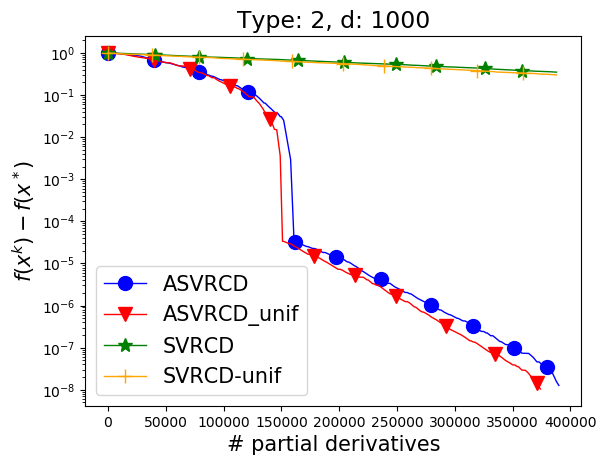}
\end{minipage}%
\\
\begin{minipage}{0.33\textwidth}
  \centering
\includegraphics[width =  \textwidth ]{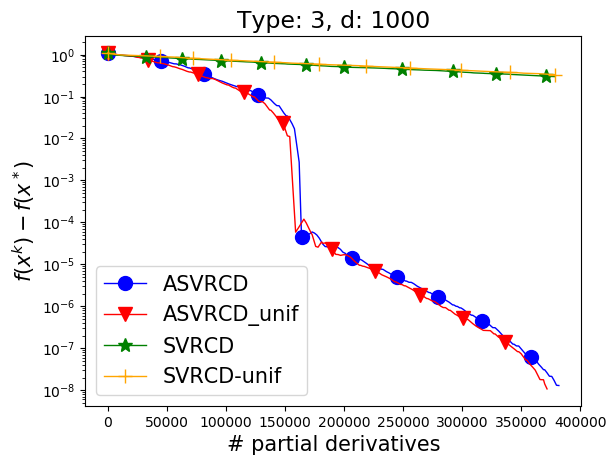}
\end{minipage}%
\begin{minipage}{0.33\textwidth}
  \centering
\includegraphics[width =  \textwidth ]{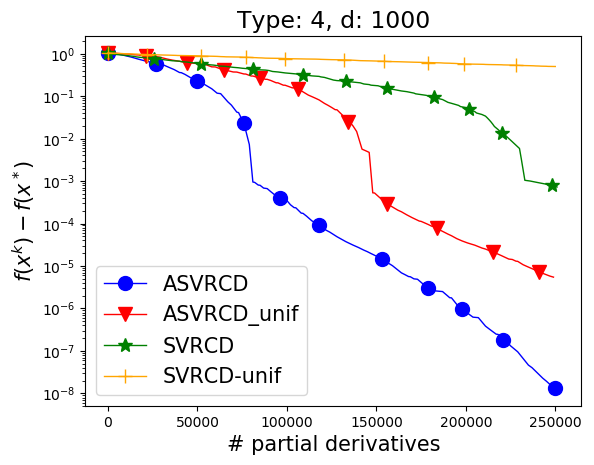}
\end{minipage}%
\caption{Comparison of both ASVRCD and SVRCD with importance and uniform sampling.} 
\label{fig:identity}
\end{figure}

The second experiment compares the performance of both ASVRCD and SVRCD for various $\mW$. We only consider methods with the importance sampling ($p_i\propto \mM_{i,i} \mW_{i,i}$) and theory supported stepsize. Figure~\ref{fig:variable} presents the result.  We see that the smaller $\Range{\mW}$ is, the faster the convergence is. This observation is well-aligned with our theory: $\ccL$ is increasing as a function of $\mW$ (in terms of Loewner ordering).

\begin{figure}[!h]
\centering
\begin{minipage}{0.33\textwidth}
  \centering
\includegraphics[width =  \textwidth ]{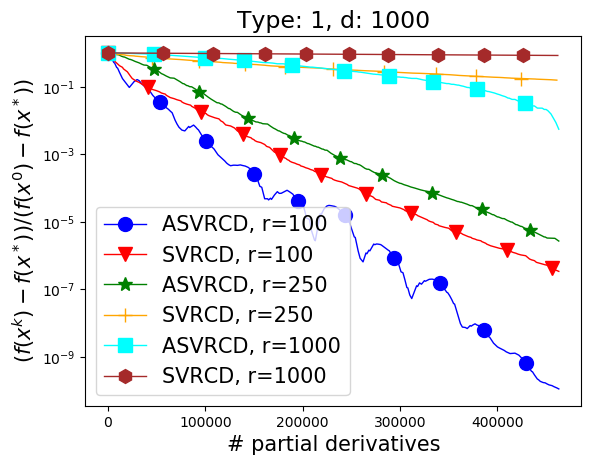}
\end{minipage}%
\begin{minipage}{0.33\textwidth}
  \centering
\includegraphics[width =  \textwidth ]{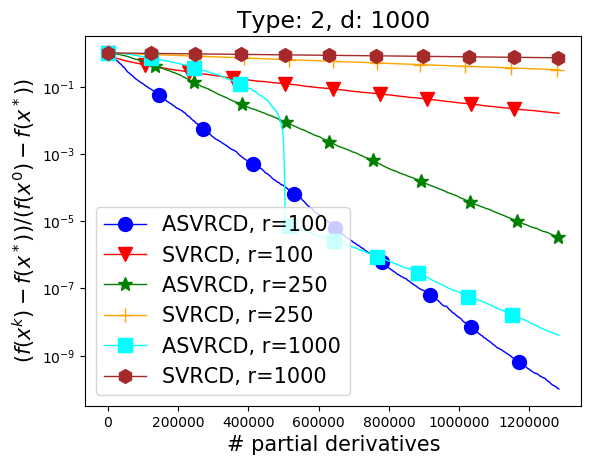}
\end{minipage}%
\\
\begin{minipage}{0.33\textwidth}
  \centering
\includegraphics[width =  \textwidth ]{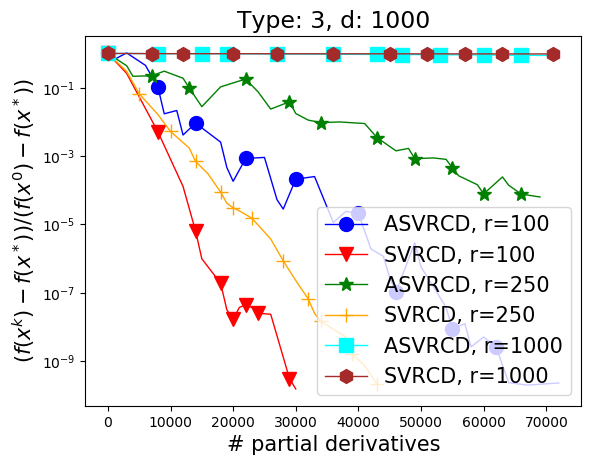}
\end{minipage}%
\begin{minipage}{0.33\textwidth}
  \centering
\includegraphics[width =  \textwidth ]{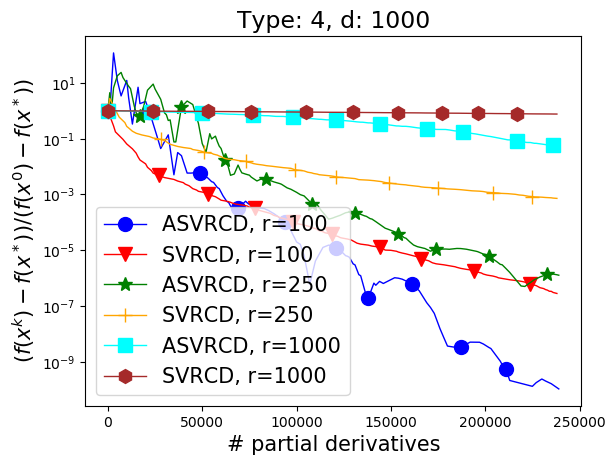}
\end{minipage}%
\caption{Comparison of ASVRCD and SVRCD for various $\Popt$. Label 'r' indicates the dimension of $\Range{\Popt}$.} 
\label{fig:variable}
\end{figure}

\section{Implications}

\paragraph{Finite sum algorithms are a special case of methods with partial derivative oracle.} Using the trick described in Sections~\ref{sec:sagasega} and~\ref{sec:kat_special}, it is possible to show that essentially any finite-sum stochastic algorithm is a special case of analogous method with partial derivative oracle (those are yet to be discovered/analyzed) in a given setting (i.e.,  strongly convex, convex, non-convex). Those include, but are not limited to SGD~\cite{robbins, nemirovski2009robust}, over-parametrized SGD~\cite{vaswani2018fast}, SAG~\cite{roux2012stochastic}, SVRG~\cite{johnson2013accelerating}, S2GD~\cite{konevcny2013semi}, SARAH~\cite{nguyen2017sarah}, incremental methods such as Finito~\cite{defazio2014finito}, MISO~\cite{mairal2015incremental} or accelerated algorithms such as point-SAGA~\cite{defazio2016simple}, Katyusha~\cite{allen2017katyusha}, MiG~\cite{zhou2018simple}, SAGA-SSNM~\cite{zhou2018direct}, Catalyst~\cite{lin2015universal, kulunchakov2019generic}, non-convex variance reduced algorithms~\cite{reddi2016stochastic, allen2016variance, fang2018spider} and others. In particular, SGD can be seen as a special case of block coordinate descent, while SAG is a special case of bias-SEGA from~\cite{hanzely2018sega} (neither of CD with non-separable prox, nor bias-SEGA were analyzed yet).

\paragraph{Zero order optimization with non-separable non-smooth regularizer.} We believe it would be interesting to develop an inexact version of ASVRCD, as this would immediately enable the application in zero-order optimization, where the partial derivatives are (inexactly) estimated using finite differences.

\subsection*{Acknowledgments}
The authors would like to express their gratitude to Konstantin Mishchenko. In particular, Konstantin has introduced us to the product space objective~\eqref{eq:equivalent} and at the same time, the idea that SAGA is a special case of SEGA was born during the discussion with him.

\bibliographystyle{plain}
 \bibliography{literature}
 \clearpage
 \appendix

\section{SAGA and L-SVRG: The algorithm}
\begin{algorithm}[!h]
	\caption{SAGA/L-SVRG}
	\label{alg:saga}
	\begin{algorithmic}
		\REQUIRE $\alpha > 0$, $\probx \in (0,1)$
		\STATE $\xx^0 \in \R^\dd, \mJ^{0} = 0 \in \R^{d\times n}$
		\FOR{$k=0,1,2,\ldots$}
		\STATE Sample random $\sS \subseteq \{1,\dots,n \}$
        \STATE{$\ggg^k = \frac1n \mJ^k \ee+ \frac1n\sum \limits_{i\in \sS}  \frac{1}{\ppp_i}(\nabla \ff_i(\xx^k) - \mJ^k_{:,i}) $}
			\STATE $\xx^{k+1} = \proxop_{\aalpha \psi} (\xx^k - \aalpha \ggg^k)$
        \STATE{Update $\mJ^{k+1}$ according to~\eqref{eq:saga_update} or~\eqref{eq:lsvrg_update}}
		\ENDFOR
	\end{algorithmic}
\end{algorithm}

 \section{Missing lemmas and proofs: SAGA/L-SVRG is a special case of SEGA/SVRCD} 
 \subsection{Proof of Lemma~\ref{lem:equivalent_objectives} \label{sec:equivalent_objectivesproof}}
 Let $\Popt' \eqdef \frac1n \ee \ee^\top \otimes \mI$ and denote $\BD(\mmM)\eqdef \blockdiag(\mmM_{1}, \dots, \mmM_n)$ for simplicity.  Now clearly $x^0 \in \Range{\Popt'}$, while $\Popt'$ is a projection matrix such that $\Ind{}(x)<\infty $ if and only if $\Popt'x=x$. Consequently, $\Popt=\Popt'$. 
 Next, if $x,y \in \Range{\Popt}$, there is $\xx, \yy \in \R^\dd$ such that $x = \Lift{\xx}, y = \Lift{\yy}$. Therefore we can write
\begin{eqnarray*}
f(x) = f(\Popt(x)) = \frac1n \sum_{j=1}^n \ff_j(\xx)& \geq&  \frac1n \sum_{j=1}^n \ff_j(\yy) + \< \nabla \left( \frac1n \sum_{j=1}^n \ff_j(\yy)\right), \xx-\yy  >  + \frac{\mmu}{2} \| \xx-\yy \|^2
\\
&=&
 f(y)+ \< \nabla f(y), x-y  >  + \frac{\mmu}{2n} \| x-y\|^2.
\end{eqnarray*}

Similarly,
\begin{eqnarray*}
f(x) = \frac1n \sum_{j=1}^n \ff_j(\xx)& \leq&  \frac1n \sum_{j=1}^n \ff_j(\yy) + \< \nabla \left( \frac1n \sum_{j=1}^n\ff_j(\yy)\right), \xx-\yy  >  + \sum_{j=1}^n \frac{1}{2n} \| \xx-\yy \|^2_{\mmM_j}
\\
&=&
 f(y)+ \< \nabla f(y), x-y  >  + \frac{1}{2n} \|x-y\|^2_{\BD(\mmM)}.
\end{eqnarray*}

Thus we conclude $\mu = \frac{ \mmu}{n}$ and $\mM = \frac1n\BD(\mmM)$.  Further, for any $h\in \R^d$, we have:
\begin{eqnarray*}
&& h^\top \mM^{\frac12} \E{  \sum_{i\in S}\pLi^{-1}\eLi \eLi^\top\Popt  \sum_{i\in S} \pLi^{-1}\eLi \eLi^\top}\mM^{\frac12}  {h}  
\\
&& \qquad  =
\frac1n h^\top \BD(\mmM)^{\frac12} \E{ \left( \sum_{i\in \sS}\pp_i^{-1}\left(\sum_{j\in R_i}e_j e_j^\top\right)\right)\Popt  \left( \sum_{i\in \sS}\pp_i^{-1}\left(\sum_{j\in R_i}e_j e_j^\top\right)\right)} \BD(\mmM)^{\frac12}  {h}
\\
&& \qquad =
 \frac{1}{n}\E{\left\| \sum_{i\in \sS }\mmM^{\frac12}_i  \pp_i^{-1}h_{R_i}   \right\|^2}
 \\
&& \qquad \stackrel{\eqref{eq:ESO_saga}}{\leq }
 \frac{1}{n}\sum_{i=1}^{n} \pp_i \vv_{i}\left\|h_{R_{i}}\right\|^{2}
\end{eqnarray*}

and thus~\eqref{eq:ESO_sega_good} holds with $v= \frac1n \vv$ as desired.

 \subsection{Proof of Lemma~\ref{lem:saga_from_sega} \label{sec:sagasegaproof}}
Denote $\Vect{\cdot}$ to be the vectorization operator, i.e.,  operator which takes a matrix as an input, and returns a vector constructed by a column-wise stacking of the matrix columns. We will show both \begin{equation}\label{eq:hj_equivalence}
h^k = \frac1n\Vect{\mJ^k}
\end{equation} 
and~\eqref{eq:iterates_equivalence} using mathematical induction. Clearly, if $k=0$ both~\eqref{eq:hj_equivalence} and~\eqref{eq:iterates_equivalence} hold. Now, let us proceed with the second induction step.

\begin{eqnarray}\nonumber
x^{k+1} 
&=& 
 \prox(x^k - \alpha g^k) = \argmin_{x\in \R^d}\,  \alpha \Ind{}(x) +  \alpha\ppsi(x_{R_1}) + \|x - (x^k - \alpha g^k )\|^2
 \\
 \nonumber
&=& 
\argmin_{x\in \R^d}\,   \alpha \Ind{}(x) + \alpha \ppsi(x_{R_1}) + \left\|x- x^k + \alpha \left( h^k+\sum \limits_{i\in S}  \frac{1}{\pLi}(\nabla_i f(x^k) - h_i^k)\eLi\right) \right  \|^2
 \\
 \nonumber
&=& 
\argmin_{x = \Popt x}\,    \alpha \ppsi(x_{R_1}) + \left\|x- x^k + \alpha \left( h^k+\sum \limits_{i\in S}  \frac{1}{\pLi}(\nabla_i f(x^k) - h_i^k)\eLi\right) \right  \|^2
 \\
 \nonumber
&=& 
\argmin_{x = \Popt x}\,    \alpha \ppsi(x_{R_1}) + \left\|x- x^k + \alpha \left( h^k+\sum \limits_{i\in S}  \frac{1}{\pLi}(\nabla_i f(x^k) - h_i^k)\eLi\right) \right  \|^2_{\Popt}
 \\
 \nonumber
&\stackrel{\eqref{eq:iterates_equivalence}}{=}& 
\Lift{\argmin_{\xx \in \R^\dd}\,    \alpha \ppsi(\xx) + \left\|\Lift{\xx}- \Lift{\xx^k} + \alpha \left( h^k+\sum \limits_{i\in \sS}  \frac{1}{\pp_i} \left(\sum_{j\in R_i}\left(\frac1n \nabla_j \ff_i(\xx^k) - h_{(i-1)\dd+j}^k\right)e_{(i-1)\dd+j}\right)\right) \right  \|^2_{\Popt}}
 \\
 \nonumber
&=& 
\Lift{\argmin_{\xx \in \R^\dd}\,    \alpha \ppsi(\xx) + \frac1n \left\|n\xx- n\xx^k + \alpha \left( \sum_{i=1}^n h^k_{R_{i}}+\sum \limits_{i\in \sS}  \frac{1}{\pp_i} \left(  \frac1n\nabla \ff_i(\xx^k) - h_{R_i }^k\right)\right)\right  \|^2}
 \\
 \nonumber
&\stackrel{\eqref{eq:hj_equivalence}}{=}& 
\Lift{\argmin_{\xx \in \R^\dd}\,    \alpha \ppsi(\xx) + \frac1n \left\|n\xx- n\xx^k + \alpha \left( \frac1n\mJ^k\ee + \frac1n\sum \limits_{i\in \sS}  \frac{1}{\pp_i} \left((\nabla\ff_i(\xx^k) -\mJ^k_{:,i})\right)\right) \right  \|^2}
 \\
 \nonumber
&=& 
\Lift{\argmin_{\xx \in \R^\dd}\,   \aalpha \ppsi(\xx) +  \left\|\xx- \xx^k + \aalpha \left( \frac1n\mJ^k\ee + \frac1n\sum \limits_{i\in \sS}  \frac{1}{\pp_i} \left((\nabla\ff_i(\xx^k) -\mJ^k_{:,i})\right)\right) \right  \|^2}
 \\
&=& 
\Lift{\xx^{k+1}}. \label{eq:sequence}
\end{eqnarray}

It remains to notice that since $x^{k+1} = \Lift{\xx^k}$, we have $h^{k+1}= \frac1n\Vect{\mJ^{k+1}}$ as desired.

 \section{Missing lemmas and proofs: ASVRCD}

\subsection{Technical lemmas}
We first start with two key technical lemmas.

\begin{lemma}
	Suppose that \begin{equation}
	\eta \leq \frac{1}{2L}. \label{eq:mkdkmd}
	\end{equation}
	Then, for all $x \in \Range{\Popt}$ the following inequality holds:
	\begin{align}
		\frac{1}{\eta} \E{\<x - x^k,x^k - y^{k+1}> }
		&\leq
		\E{P(x) - P(y^{k+1}) - \frac{1}{4\eta}\norm{y^{k+1} - x^k}^2 + \frac{\eta}{2}\norm{g^k - \nabla f(x^k)}^2_\Popt}-D_f(x,x^k). \label{eq:keylemma_acc}
	\end{align}
\end{lemma}

\begin{proof}
	From the definition of $y^{k+1}$ we get
	\begin{equation*}
		y^{k+1} = x^k - \eta g^k - \eta \Delta,
	\end{equation*}
	where $\Delta \in \partial \psi(y^{k+1})$.
	Therefore,
	\begin{align}\nonumber
		\E{\frac{1}{\eta}\<x - x^k, x^k - y^{k+1}>}
     &	=
		\E{\<x - x^k, g^k + \Delta>}
		\\ \nonumber
		&=
		\<x - x^k, \nabla f(x^k)>
		+
		\E{\<x - y^{k+1}, \Delta> + \< y^{k+1} - x^k, \Delta>}\\
		&\leq
		f(x) - f(x^k) - D_f(x,x^k)
		+
		\E{\psi(x) - \psi(y^{k+1})}
		+
		\E{\< y^{k+1} - x^k, \Delta>} \label{eq:dnjansdjkajksd}
	\end{align}
	Now, we use the fact that $f$ is $L$-smooth over the set where iterates live (i.e.,  over $\{ x^0 + \Range{\Popt}\}$):
	\begin{eqnarray}\nonumber
			f(y^{k+1})& \leq& f(x^k) + \<\nabla f(x^k), y^{k+1} - x^k> + \frac{L}{2}\norm{y^{k+1} - x^k}^2 \\
		& = & f(x^k) + \<\Popt\nabla f(x^k), y^{k+1} - x^k> + \frac{L}{2}\norm{y^{k+1} - x^k}^2.\label{eq:dabhusdbhu}
	\end{eqnarray}
Thus, we have
	\begin{eqnarray*}
&&	\E{\frac{1}{\eta}\<x - x^k, x^k - y^{k+1}>} \\
	&& \,\,\,\,\,  \stackrel{\eqref{eq:dnjansdjkajksd}+\eqref{eq:dabhusdbhu}}{\leq}
	\E{P(x) - P(y^{k+1})
	+
	\< y^{k+1} - x^k, \Popt(\Delta + \nabla f(x^k))>
	+
	\frac{L}{2}\norm{y^{k+1} - x^k}^2} \\
	&& \qquad \qquad - D_f(x,x^k)
	 \\
	&& \qquad =
	\E{P(x) - P(y^{k+1})
		+
		\< y^{k+1} - x^k, \Popt(\nabla f(x^k) - g^k)>
		-
		\frac{1}{\eta}\norm{y^{k+1} - x^k}^2}
\\
&& \qquad
 \qquad 		
		+\E{\frac{L}{2}\norm{y^{k+1} - x^k}^2}- D_f(x,x^k) \\
 	&& \qquad \leq
	\E{P(x) - P(y^{k+1})
		+
		\frac{\eta}{2}\norm{\nabla f(x^k) - g^k}^2_\Popt
		-
		\frac{1}{2\eta}\norm{y^{k+1} - x^k}^2
		+
		\frac{L}{2}\norm{y^{k+1} - x^k}}\\
&& \qquad
  \qquad - D_f(x,x^k)\\
 	&& \qquad \stackrel{\eqref{eq:mkdkmd}}{\leq}
		\E{P(x) - P(y^{k+1})
		-
		\frac{1}{4\eta}\norm{y^{k+1} - x^k}^2
		+
		\frac{\eta}{2}\norm{\nabla f(x^k) - g^k}^2_\Popt} - D_f(x,x^k),  \\
	\end{eqnarray*}
	 	which concludes the proof. 
\end{proof}

\begin{lemma}
	Suppose, the following choice of parameters is used:
	\begin{equation*}
		\eta =  \frac14 \max\{\cL, L\}^{-1},\qquad
		\gamma = \frac{1}{\max\{2\mu, 4\theta_1/\eta\}},\qquad
		\beta = 1 - \gamma\mu,\qquad
		\theta_2 = \frac{\cL}{2\max\{L, \cL\}}.
	\end{equation*}
	Then the following inequality holds:
	\begin{align}\nonumber
		&\E{\norm{z^{k+1} - x^*}^2 + \frac{2\gamma\beta}{\theta_1}\left[P(y^{k+1}) - P(x^*)\right]}\\
		& \qquad \leq
		\beta \norm{z^k - x^*}^2
		+
		\frac{2\gamma\beta\theta_2}{\theta_1}\left[P(w^k)  - P(x^*)\right]
		+
		\frac{2\gamma\beta(1-\theta_1-\theta_2)}{\theta_1}
		\left[P(y^k) - P(x^*)\right].\label{eq:nhivbhi}
	\end{align}
\end{lemma}
\begin{proof}
	\begin{align*}
		\E{\norm{z^{k+1} - x^*}^2}
		&=
		\E{\norm{\beta z^k + (1-\beta)x^k - x^* + \frac{\gamma}{\eta}(y^{k+1} - x^k)}^2}
		\\
		&\leq
		\beta \norm{z^k - x^*}^2
		+
		(1-\beta)\norm{x^k - x^*}^2
		+
		\frac{\gamma^2}{\eta^2}\E{\norm{y^{k+1} - x^k}^2}
		\\
		& \qquad +
		\frac{2\gamma}{\eta}\E{\<y^{k+1} - x^k, \beta z^k + (1-\beta)x^k - x^*>}
		\\
		&=
		\beta \norm{z^k - x^*}^2
		+
		(1-\beta)\norm{x^k - x^*}^2
		+
		\frac{\gamma^2}{\eta^2}\E{\norm{y^{k+1} - x^k}^2}
		\\
		& \qquad +
		\frac{2\gamma}{\eta}\E{\<y^{k+1} - x^k, x^k - x^*>}
		+
		\frac{2\gamma\beta}{\eta}\E{\<y^{k+1} - x^k, z^k - x^k>}
		\\
		&=	
		\beta \norm{z^k - x^*}^2
		+
		(1-\beta)\norm{x^k - x^*}^2
		+
		\frac{\gamma^2}{\eta^2}\E{\norm{y^{k+1} - x^k}^2}
		+
		\frac{2\gamma}{\eta}\E{\<x^k - y^{k+1}, x^* - x^k>}
		\\
		& \qquad +
		\frac{2\gamma\beta\theta_2}{\eta\theta_1}\E{\<x^k - y^{k+1} , w^k - x^k>}
		+
		\frac{2\gamma\beta( 1  - \theta_1 -\theta_2 )}{\eta\theta_1}\E{\<x^k - y^{k+1} , y^k - x^k>}
		\\
		&\stackrel{\eqref{eq:keylemma_acc}}{\leq}
		\beta \norm{z^k - x^*}^2
		+
		(1-\beta)\norm{x^k - x^*}^2
		+
		\frac{\gamma^2}{\eta^2}\E{\norm{y^{k+1} - x^k}^2}
		\\
		&\qquad		+
		2\gamma \E{P(x^*) - P(y^{k+1}) - \frac{1}{4\eta}\norm{y^{k+1} - x^k}^2 - D_f(x^*,x^k) + \frac{\eta}{2}\norm{g^k-\nabla f(x^k)}^2_\Popt}
		\\
		& \qquad +
		\frac{2\gamma\beta\theta_2}{\theta_1}\E{P(w^k) - P(y^{k+1}) - \frac{1}{4\eta}\norm{y^{k+1} - x^k}^2 - D_f(w^k,x^k) + \frac{\eta}{2}\norm{g^k-\nabla f(x^k)}^2_\Popt}
		\\
		&\qquad +
		\frac{2\gamma\beta( 1  - \theta_1 -\theta_2 )}{\theta_1}\E{P(y^k) - P(y^{k+1}) - \frac{1}{4\eta}\norm{y^{k+1} - x^k}^2 + \frac{\eta}{2}\norm{g^k - \nabla f(x^k)}^2_\Popt}
		\\
		&\stackrel{\eqref{eq:sc}}{\leq}
		\beta \norm{z^k - x^*}^2
		+
		(1-\beta - \gamma\mu)\norm{x^k - x^*}^2
		+
		\frac{\gamma^2}{\eta^2}\E{\norm{y^{k+1} - x^k}^2}
		\\
		&\qquad 		+
		2\gamma\beta \E{P(x^*) - P(y^{k+1}) - \frac{1}{4\eta}\norm{y^{k+1} - x^k}^2} +  \eta\gamma\E{\norm{g^k-\nabla f(x^k)}^2_\Popt}
		\\
		& \qquad+
		\frac{2\gamma\beta\theta_2}{\theta_1}\E{P(w^k) - P(y^{k+1}) - \frac{1}{4\eta}\norm{y^{k+1} - x^k}^2 -  D_f(w^k,x^k)+ \frac{\eta}{2}\norm{g^k-\nabla f(x^k)}^2_\Popt}
		\\
		& \qquad +
		\frac{2\gamma\beta(1  - \theta_1 -\theta_2)}{\theta_1}\E{P(y^k) - P(y^{k+1}) - \frac{1}{4\eta}\norm{y^{k+1} - x^k}^2 + \frac{\eta}{2}\norm{g^k - \nabla f(x^k)}^2_\Popt}.
	\end{align*}
		Using $\beta = 1 - \gamma\mu$ we get
	\begin{align*}
	\E{\norm{z^{k+1} - x^*}^2}
	&\leq
	\beta \norm{z^k - x^*}^2
	+
	\left[\frac{\gamma^2}{\eta^2} - \frac{\gamma\beta}{2\eta\theta_1}\right]\E{\norm{y^{k+1} - x^k}^2}
	+
	\frac{\eta\gamma}{\theta_1}\E{\norm{g^k - \nabla f(x^k)}^2_\Popt}
		\\
	&
	\qquad -
	\frac{2\gamma\beta\theta_2}{\theta_1}D_f(w^k,x^k)
	+
	2\gamma\beta \E{P(x^*) - P(y^{k+1})}\\
	&\qquad
	+
	\frac{2\gamma\beta\theta_2}{\theta_1}\E{P(w^k) - P(y^{k+1})}
	+
	\frac{2\gamma\beta(1-  \theta_1 -\theta_2 )}{\theta_1}\E{P(y^k) - P(y^{k+1})}.
	\end{align*}
	Using stepsize $\gamma \leq \frac{\beta\eta}{2\theta_1}$ we get
	\begin{align*}
	\E{\norm{z^{k+1} - x^*}^2}
	&\leq
	\beta \norm{z^k - x^*}^2
	+
	\frac{\eta\gamma}{\theta_1}\E{\norm{g^k - \nabla f(x^k)}^2_\Popt}
	-
	\frac{2\gamma\beta\theta_2}{\theta_1}D_f(w^k,x^k)
		+
	2\gamma\beta \E{P(x^*) - P(y^{k+1})}
	\\
	&
	\qquad +
	\frac{2\gamma\beta\theta_2}{\theta_1}\E{P(w^k) - P(y^{k+1})}
	+
	\frac{2\gamma\beta(1  - \theta_1 - \theta_2)}{\theta_1}\E{P(y^k) - P(y^{k+1})}.
	\end{align*}
	Now, using the expected smoothness from inequality~\eqref{eq:exp:smooth}:
	\begin{equation}
		\E{\norm{g^k - \nabla f(x^k)}^2_\Popt} \leq 2\cL D_f(w^k,x^k)
	\end{equation}
	 and stepsize $\eta \leq \frac{\beta\theta_2}{\cL}$ we get
	 \begin{align*}
	 \E{\norm{z^{k+1} - x^*}^2}
	 &\leq
	 \beta \norm{z^k - x^*}^2
   +
	 \frac{2\cL\eta\gamma }{\theta_1} D_f(w^k,x^k)
	 -
	 \frac{2\gamma\beta\theta_2}{\theta_1} D_f(w^k,x^k)
	 +
	 2\gamma\beta \E{P(x^*) - P(y^{k+1})}
	 \\
	 &
	 \qquad 
	 +
	 \frac{2\gamma\beta\theta_2}{\theta_1}\E{P(w^k) - P(y^{k+1})}
	 +
	 \frac{2\gamma\beta(1 - \theta_1- \theta_2 )}{\theta_1}\E{P(y^k) - P(y^{k+1})}\\
	 &\leq
	  \beta \norm{z^k - x^*}^2
	  +
	  2\gamma\beta \E{P(x^*) - P(y^{k+1})}
	  +
	  \frac{2\gamma\beta\theta_2}{\theta_1}\E{P(w^k) - P(y^{k+1})}\\
	  & \qquad +
	  \frac{2\gamma\beta( 1 - \theta_1- \theta_2 )}{\theta_1}\E{P(y^k) - P(y^{k+1})}\\
	  &=
	  \beta \norm{z^k - x^*}^2
	  -
	  \frac{2\gamma\beta}{\theta_1}\E{P(y^{k+1}) - P(x^*)}
	  +
	  \frac{2\gamma\beta\theta_2}{\theta_1}\left[P(w^k)  - P(x^*)\right] \\
	  & \qquad +
	  \frac{2\gamma\beta( 1-\theta_1-\theta_2)}{\theta_1}
	  \left[P(y^k) - P(x^*)\right].
	 \end{align*}
	 It remains to rearrange the terms.
\end{proof}

\subsection{Proof of Theorem~\ref{thm:acc}} 
 
	One can easily show that
	\begin{equation}\label{eq:bfrbuf}
		\E{P(w^{k+1})} = \probx P(y^k) + (1- \probx)P(w^k).
	\end{equation}
	Using that we obtain
	\begin{eqnarray*}
		\E{\Psi^{k+1}}
		&\stackrel{\eqref{eq:nhivbhi}+\eqref{eq:bfrbuf}}{\leq}&
		\beta \norm{z^k - x^*}^2
		+
		\frac{2\gamma\beta\theta_2}{\theta_1}\left[P(w^k)  - P(x^*)\right]
		+
		\frac{2\gamma\beta( 1-\theta_1 -\theta_2)}{\theta_1}
		\left[P(y^k) - P(x^*)\right]\\
		&& \qquad +
		\frac{(2\theta_2 + \theta_1)\gamma\beta}{\theta_1\probx}\left[ \probx P(y^{k}) + (1- \probx)P(w^k) - P(x^*)\right]\\
		&=&
		\beta \norm{z^k - x^*}^2
		+
		\frac{2\gamma\beta(1- \theta_1/2)}{\theta_1}
		\left[P(y^k) - P(x^*)\right]
		\\
		&& \qquad 
		+
		\frac{(2\theta_2 + \theta_1)\gamma\beta}{\theta_1\probx}\left[1 -  \probx + \frac{2\probx\theta_2}{2\theta_2 + \theta_1}\right]\left[P(w^k) - P(x^*)\right]\\
		&\leq&
		\max\left\{1 - \frac{1}{\max\{2, 4\theta_1/(\eta\mu)\}}, 1 -  \frac{\theta_1}{2}, 1 - \frac{ \probx\theta_1}{2\max \{2\theta_2,\theta_1\}} \right\}\Psi^k\\
		&=&
		 \left[1 -  \max\left\{\frac{2}{\probx},\frac{4}{\theta_1}\max\left\{\frac{1}{2}, \frac{\theta_2}{\rho}\right\}, \frac{4\theta_1}{\eta\mu}    \right\}^{-1} \right]\Psi^k.
	\end{eqnarray*}
	Using $\theta_1 = \min\left\{\frac{1}{2},\sqrt{\eta\mu \max\left\{\frac{1}{2}, \frac{\theta_2}{\rho}\right\}}\right\} $ we get
	\begin{align*}
		\E{\Psi^{k+1}} &\leq
		\left[1 -  \max\left\{\frac{2}{\probx},8\max\left\{\frac{1}{2}, \frac{\theta_2}{\rho}\right\}, 4\sqrt{\frac{\max\left\{\frac{1}{2}, \frac{\theta_2}{\rho}\right\}}{\eta\mu}} \right\}^{-1} \right]\Psi^k\\
		&\leq
		\left[1 -  \frac{1}{4}\max\left\{\frac{1}{\probx}, \sqrt{\frac{2\max\left\{L, \frac{\cL}{\rho}\right\}}{\mu}} \right\}^{-1} \right]\Psi^k,
	\end{align*}
as desired.

\subsection{Proof of Lemma~\ref{lem:exp_smooth_vs_ESO}}

To establish that that we can choose $\cL = \ccL$, it suffices to see
\begin{eqnarray*}
\E{\norm{g^k - \nabla f(x^k)}^2_\Popt}  &= & \E{\norm{  \sum \limits_{i\in S}  \frac{1}{\pLi}(\nabla_i f(x^k) - \nabla_i f(w^k))\ones_i + \nabla f(w^k) - \nabla f(x^k)}^2_\Popt}  \\
& \leq  & 
 \E{\norm{  \sum \limits_{i\in S}  \frac{1}{\pLi}(\nabla_i f(x^k) - \nabla_i f(w^k))\ones_i}^2_\Popt}   \\
   &\stackrel{\eqref{eq:ccLdef}}{\leq} & 
\ccL \norm{ \nabla f(x^k) - \nabla f(w^k) }^2_{\mM^{-1}}\\
     &\stackrel{\eqref{eq:smooth}}{\leq} & 
 2 \ccL D_f(w^k,x^k).
\end{eqnarray*}

Next, to establish $\ccL \geq L $, let $\mQ \eqdef   \sum_{i\in S} \frac{1}{\pLi} \eLi \eLi^\top\Popt$. Consequently, we get

\begin{eqnarray*}
\ccL &\stackrel{\eqref{eq:ccLdef}}{=}&
 \lambda_{\max}\left( \mM^{\frac12} \E{  \sum_{i\in S} \frac{1}{\pLi} \eLi \eLi^\top\Popt \sum_{i\in S} \frac{1}{\pLi} \eLi \eLi^\top} \mM^{\frac12} \right)
 \\
 &=&
  \lambda_{\max}\left( \mM^{\frac12} \E{  \sum_{i\in S} \frac{1}{\pLi} \eLi \eLi^\top\Popt^2 \sum_{i\in S} \frac{1}{\pLi} \eLi \eLi^\top} \mM^{\frac12} \right)
 \\
 & =&
  \lambda_{\max}\left( \mM^{\frac12} \E{  \mQ \mQ^\top} \mM^{\frac12} \right)
 \\
& \geq &
   \lambda_{\max}\left( \mM^{\frac12} \E{  \mQ}\E{ \mQ^\top} \mM^{\frac12} \right)
 \\
  &=&
  \lambda_{\max}\left( \mM^{\frac12}\Popt^2 \mM^{\frac12} \right)
   \\
  &=&
  \lambda_{\max}\left( \mM^{\frac12}\Popt \mM^{\frac12} \right)
     \\
  &=&
 L,
\end{eqnarray*}
as desired. 

\subsection{Proof of Lemma~\ref{lem:acc_example}}
Let us look first at $\Popt = \mI$. In such case, it is easy to see that 
\begin{align*}
\E{\norm{g^k - \nabla f(x^k)}^2_{\Popt}} &= \E{\norm{g^k - \nabla f(x^k)}^2} \\
 &\leq  \E{\norm{d ( \nabla_i f(x^k) - \nabla_i f(w^k) ) e_i }^2}   \\
&= d \norm{ \nabla f(x^k) - \nabla f(w^k) }^2 \\
&\leq  2d \lambda_{\max} \mM D_f(w^k,x^k),
\end{align*}
  i. e. we can choose $\cL = d \lambda_{\max} \mM$. Noting that $ \lambda_{\max} \mM\geq L$, the iteration complexity of Algorithm~\ref{alg:acc} is $\cO\left(d\sqrt{\frac{L}{\mu}} \log \frac1\epsilon\right)$.

On the other hand, if $\Popt = \frac1d \ee^\top$, we have 
\begin{align*}
\E{\norm{g^k - \nabla f(x^k)}^2_{\Popt}} =&\E{\norm{g^k - \nabla f(x^k)}^2_{\frac1d ee^\top}} \\
 &  \leq  \E{\norm{d ( \nabla_i f(x^k) - \nabla_i f(w^k) ) e_i }^2_{\frac1d ee^\top}}   \\
&=  \norm{ \nabla f(x^k) - \nabla f(w^k) }^2 \\
&\leq  2 \lambda_{\max} \mM D_f(w^k,x^k),
\end{align*}
and therefore $\cL =  \lambda_{\max} \mM$, which yields $\cO\left(\sqrt{\frac{d \lambda_{\max} \mM}{\mu}} \log \frac1\epsilon\right)$ convergence rate.

 \section{Missing lemmas and proofs: L-Katyusha as a particular case of ASVRCD}

\subsection{Proof of Lemma~\ref{lem:katyusha_from_asvrcd}}
Let us proceed by induction. We will show the following for all $k\geq 0$ we have 
\begin{equation} \label{eq:induction_acc}
\xx^k = x^k_{R_1}=\dots =x^k_{R_n}, \; \yy^k= y^k_{R_1}=\dots =y^k_{R_n},\;  \zz^k = z^k_{R_1}=\dots =z^k_{R_n} \; \mathrm{and} \: \ww^k = w^k_{R_1}=\dots =w^k_{R_n}.
\end{equation}

Clearly, for $k=0$, the above claim holds. Let us proceed with the second induction step and assume that~\eqref{eq:induction_acc} holds for some $k\geq 0$. First, the update rule on $\{x^k\}$ for ASVRCD together with the update rule on $\{\xx^k\}$ yields 
\begin{equation}\label{eq:induction_acc_x}
\xx^{k+1} = x^{k+1}_{R_1}=\dots= x^{k+1}_{R_n}. 
\end{equation}

To show 
\begin{equation}\label{eq:induction_acc_y}
\yy^{k+1}= y^{k+1}_{R_1}=\dots =y^{k+1}_{R_n},
\end{equation}

we essentially repeat the proof of Lemma~\ref{lem:saga_from_sega}. In particular, it is sufficient to repeat the sequence of inequalities~\eqref{eq:sequence} where variables \[(x^{k+1}, \xx^{k+1}, h^k, \mJ^{k} \alpha, \aalpha)\] are replaced by \[(y^{k+1}, \yy^{k+1},\nabla f(w^k), [\nabla \ff_1(\ww^k), \dots,\nabla \ff_n(\ww^k) ], \eta, \eeta)\] respectively. 

Next, $ \zz^{k+1}= z^{k+1}_{R_1}=\dots= z^{k+1}_{R_n}$ follows from~\eqref{eq:induction_acc},~\eqref{eq:induction_acc_x} and~\eqref{eq:induction_acc_y} together with the update rule (on $\{z^k\}$ and $\{\zz^k\}$) of both algorithms and the fact that $\frac{\gamma}{\eta} = \frac{\ggamma}{\eeta}$.

To finish the proof of the algorithms equivalence, we shall notice that $ \ww^{k+1}= w^{k+1}_{R_1}=\dots = w^{k+1}_{R_n}$ follows from~\eqref{eq:induction_acc},~\eqref{eq:induction_acc_y} together with the update rule (on $\{w^k\}$ and $\{\ww^k\}$) of both algorithms.

To show $\cL = \frac{\cLL}{n} $ it is sufficient to see
\begin{eqnarray*}
\E{\norm{g^k - \nabla f(x^k)}^2_\Popt} 
&=& 
\E{  
\norm{\sum_{i\in \sS} p_{i}^{-1}\left( \sum_{j\in R_i} \left( \nabla_{j} f(x^k) - \nabla_{j} f(w^k)\right) e_j \right) - \left(\nabla f(x^k) - \nabla f(w^k)\right)}^2_\Popt
}
\\
&=&
\E{  
\norm{\Popt \left(\sum_{i\in \sS} p_{i}^{-1}\left( \sum_{j\in R_i} \left( \nabla_{j} f(x^k) - \nabla_{j} f(w^k)\right) e_j \right)\right) - \Popt\left(\nabla f(x^k) - \nabla f(w^k)\right)}^2
}
\\
&=&
\frac1n \E{  
\norm{ \left( \frac1n \sum_{i\in \sS} \pp_{i}^{-1} \left( \nabla \ff_i(\xx^k) - \nabla \ff_i(\ww^k)\right) \right) -  \left(\nabla \ff(\xx^k) - \nabla \ff(\ww^k)\right)}^2
}
\\
&=&
\frac1n\E{  
\norm{ \ggg^k-  \nabla \ff(\ww^k)}^2
}
\\
&\leq&
 2\frac{\cLL}{n} D_\ff(\ww^k,\xx^k)
 \\
&=&
 2\frac{\cLL}{n} \left( \frac1n\sum_{i=1}^n D_{\ff_i}(w^k_{R_i},x^k_{R_i})\right)
  \\
&=&
 2\frac{\cLL}{n} D_{f}(w^k,x^k).
\end{eqnarray*}

Lastly, if $x,y \in \Range{\Popt}$, there is $\xx, \yy \in \R^\dd$ such that  $x = \Lift{\xx}, y = \Lift{\yy}$. Therefore we can write
\begin{eqnarray*}
f(x) = f(\Popt(x)) = \frac1n \sum_{j=1}^n \ff_j(\xx)& \leq&  \frac1n\sum_{j=1}^n \ff_j(\yy) + \< \nabla \left( \frac1n \sum_{j=1}^n \ff_j(\yy)\right), \xx-\yy  >  + \frac{\LL}{2} \| \xx-\yy \|^2
\\
&=&
 f(y)+ \< \nabla f(y), x-y  >  + \frac{\LL}{2n} \| x-y\|^2
\end{eqnarray*}
 and thus $L = \lambda_{\max} \left( \mM^{\frac12} \Popt \mM^{\frac12}\right)  \leq  n^{-1}\LL$.

 \section{Tighter rates for GJS~\cite{hanzely2019one} by exploiting prox \label{sec:analysis2} and proof of Theorem~\ref{thm:sega_as}}

In this section, we show that specific nonsmooth function $\psi$ might lead to faster convergence of variance reduced methods. We exploit the well-known fact that under some circumstances, a proximal operator might change the smoothness structure of the objective~\cite{gutman2019condition}. In particular, we consider Generalized Jacobian Sketching (GJS) from~\cite{hanzely2019one}. We generalize Theorem 5.1 therein, which allows for a tighter rate if $\psi$ has a specific structure. 

\subsection{GJS}

Consider a the following objective:
\[
\min_{x\in \R^d} \sum_{i=1}^n f_i(x) + \psi(x).
\]
and define Jacobian operator $\mG: \R^{d}\rightarrow \R^{d\times n}$ as $\mG(x) \eqdef [ \nabla f_1(x), \dots,  \nabla f_n(x)]$. Further, define $\cM: \R^{d\times n} \rightarrow \R^{d\times n}$ to be such linear operator that the following holds $\left(\cM \mX\right)_{:j} = \mM_j \mX_{:j}$ for $j\in [n]$.

 Suppose that $\cU: \R^{d\times n} \rightarrow \R^{d\times n}$ is a random linear operator such that $\E{\cU}$ is identity, and $\cS: \R^{d\times n} \rightarrow \R^{d\times n}$ is a random projection operator. Given the (fixed) distribution over $\cU, \cS$, GJS is a variance reduced algorithm with the oracle access to $\cU \mG(x), \cS \mG(x)$. 

\begin{algorithm}[!h]
\begin{algorithmic}[1]
\STATE \textbf{Parameters:} Stepsize $\alpha>0$, random projector $\cS$ and unbiased sketch $\cU$
\STATE \textbf{Initialization:} Choose  solution estimate $x^0 \in \R^d$ and Jacobian estimate $ \mJ^0\in \R^{d\times n}$ 
\FOR{$k =  0, 1, \dots$}
\STATE Sample realizations of $\cS$ and $\cU$, and perform sketches $\cS\mG(x^k)$ and $\cU\mG(x^k)$
\STATE  $\mJ^{k+1} = \mJ^k - \cS(\mJ^k - \mG(x^k))$ \quad \hfill update the Jacobian estimate
\STATE $g^k = \frac1n \mJ^k \eR + \frac1n \cU \left(\mG(x^k) -\mJ^k\right)\eR$  \hfill construct the gradient estimator 
    \STATE $x^{k+1} = \proxgjs (x^k - \alpha g^k)$ \label{eq:alg_update} \hfill perform the proximal SGD step 
\ENDFOR
\end{algorithmic}
\caption{Generalized JacSketch (GJS)}
\label{alg:SketchJac}
\end{algorithm}

\begin{theorem}[Extension of Theorem 5.1 from~\cite{hanzely2019one}] \label{thm:main2}

Define $f(x) \eqdef \frac1n \sum_{i=1}^n f_i(x)$. Let Assumption~\ref{ass:indicator} hold and suppose that  ${\cM^\dagger}^{\nicefrac{1}{2}}$ commutes with $\cS$. Next, let $\alpha$ and $\cB$ are such that for every $\mX\in \R^{d\times n}$ we have
\begin{equation}\label{eq:small_step_v2}
 \frac{2\alpha}{n^2}  \E{ \norm{ \cU  \mX e }^2_{\Popt} }   +   \NORMG{  \left(\cI - \E{\cS} \right)^{\frac12}\cB  {\cM^\dagger} \mX }  \leq (1-\alpha \sigma') \NORMG{ \cB {\cM^\dagger}\mX },
\end{equation}
\begin{equation}
\frac{2\alpha}{n^2} \E{  \norm{ \cU  \mX  e }_{\Popt}^2  } 
+   \NORMG{\left(\E{\cS}\right)^{\frac12}  \cB  {\cM^\dagger} \mX  }    \leq \frac{1}{n} \norm{{\cM^\dagger}\mX}^2\label{eq:small_step2_v2}
\end{equation}
and $\cB$ commutes with $\cS$. Then  for all $k\geq 0$, we have
$\E{\Psi^{k}}\leq \left( 1-\alpha\sigma'\right)^k \Psi^{0},$ where
\begin{eqnarray*}\label{eq:Lyapunov}
\Psi^k & \eqdef & \norm{ x^k - x^* }^2 + \alpha \NORMG{ \cB {\cM^\dagger}^{\frac12} \left( \mJ^k - \mG(x^*)\right)}.
\end{eqnarray*}
\end{theorem}

\subsection{Towards the proof of Theorem~\ref{thm:main2}}

\begin{lemma} (Slight extension of Lemma from~\cite{hanzely2019one})\label{lem:g_lemma}
Let $\cU$ be random linear operator which is identity in expectation. Let $\mG(x)$ be Jacobian at $x$ and $g^k = \frac1n \cU( \mG(x) -\mJ^k)\eR   -\frac1n\mJ^k\eR$.
Then for any $\mQ \in \R^{d\times d}, \mQ\succeq 0$ and all $k\geq 0$ we have
\begin{equation}\label{eq:g_lemma}
\E{ \norm{ g^k - \nabla f(x^*)}_\mQ^2} \leq    \frac{2}{n^2} \E{  \norm{ \cU \left(\mG(x^k) - \mG(x^*) \right) \eR}_\mQ^2  } + \frac{2}{n^2}  \E{ \norm{ \cU \left(\mJ^k - \mG(x^*) \right) \eR}_\mQ^2}.
\end{equation}
\end{lemma}
\begin{proof}
Since $\nabla f(x^*) = \frac{1}{n}\mG(x^*) \eR$, we have
\begin{equation}\label{eq:nb87fvdbs8s}
g^k-\nabla f(x^*) = \underbrace{\frac1n \cU \left(\mG(x^k)- \mG(x^*) \right) \eR}_{a}  +  \underbrace{\frac1n \left(\mJ^k - \mG(x^*) \right) \eR - \frac1n \cU \left( \mJ^k - \mG(x^*) \right) \eR}_{b} .
\end{equation}
Applying the bound $\norm{a+b}_{\mQ}^2 \leq 2\norm{a}_{\mQ}^2 + 2\norm{b}_{\mQ}^2$ to \eqref{eq:nb87fvdbs8s} and taking expectations, we get
\begin{eqnarray*}
\E{ \norm{ g^k -\nabla f(x^*) }_\mQ^2} &\leq &
 \E{ \frac{2}{n^2} \norm{  \cU \left(\mG(x^k)- \mG(x^*) \right) \eR  }_\mQ^2} \\
 && \qquad + \E{ \frac{2}{n^2} \norm{   \left(\mJ^k-\mG(x^*) \right) \eR -   \cU \left(\mJ^k - \mG(x^*) \right) \eR }_\mQ^2  }\\
&=&
 \frac{2}{n^2} \E{  \norm{ \cU \left(\mG(x^k) - \mG(x^*) \right) \eR }_\mQ^2  } \\
 && \qquad + \frac{2}{n^2} \E{  \norm{ \left(\cI- \cU \right) \left(\mJ^k - \mG(x^*) \right) \eR }_\mQ^2 }.
\end{eqnarray*}

It remains to note that
\begin{eqnarray*}
\E{  \norm{ \left(\cI-  \cU) (\mJ^k - \mG(x^*) \right) \eR }_\mQ^2 }
 &=&
\E{ \norm{ \cU \left(\mJ^k - \mG(x^*) \right) \eR  }_{\mQ}^2 } - \norm{ \left( \mJ^k - \mG(x^*) \right) \eR }_\mQ^2
  \\
  &\leq& 
  \E{ \norm{ \cU \left(\mJ^k - \mG(x^*) \right) \eR }_\mQ^2 } .
\end{eqnarray*}

\end{proof}

 \begin{lemma} (\cite{hanzely2019one}, Lemma E.3) \label{lem:smooth_consequence} Assume that function $f_j$ are convex and $\mM_j$-smooth. Then
\begin{equation}\label{eq:smooth}
D_{f_j}(x,y) \geq \frac12 \norm{ \nabla f_j(x)-\nabla f_j(y) }^2_{\mM_j^{\dagger}}, \quad \forall x,y\in \R^d, 1\leq j \leq n.
\end{equation}
If $x-y\in \Null{\mM_j}$, then 
\begin{enumerate} 
\item[(i)]  \begin{equation}\label{eq:linear_on_subspace} f_j(x) = f_j(y) + \langle \nabla f_j(y), x-y\rangle,\end{equation}
\item[(ii)]
\begin{equation} \label{eq:n98g8ff} \nabla f_j(x)-\nabla f_j(y) \in \Null{\mM_j},\end{equation}
\item[(iii)]  
\begin{equation} \label{eq:nb87sgb} \langle \nabla f_j(x) - \nabla f_j(y),x-y\rangle =0.\end{equation}
\end{enumerate}

If, in addition, $f_j$ is bounded below, then $\nabla f_j(x)  \in \Range{\mM_j}$ for all $x$.
\end{lemma}

\begin{lemma}(\cite{hanzely2019one}, Lemma E.5) \label{lem:nb98gd8fdx}
Let $\cS$ be a random projection operator and $\cA$ any deterministic linear operator commuting with $\cS$, i.e.,  $\cA \cS = \cS \cA$. Further, let $\mX,\mY \in \R^{d\times n}$ and define $\mZ = (\cI-\cS) \mX + \cS \mY$. Then
\begin{itemize}
\item[(i)] $\cA \mZ = (\cI-\cS) \cA \mX + \cS \cA \mY $,
\item[(ii)] $\norm{\cA \mZ}^2 = \norm{(\cI-\cS) \cA \mX}^2  + \norm{\cS \cA \mY}^2 $,
\item[(iii)] $\E{\norm{\cA \mZ}^2} = \norm{(\cI-\E{\cS})^{1/2} \cA \mX}^2  + \norm{\E{\cS}^{1/2} \cA \mY}^2 $, where the expectation is with respect to $\cS$.
\end{itemize}
\end{lemma}

\paragraph{Proof of Theorem~\ref{thm:main2}}
For simplicity of notation, in this proof, all expectations are conditional on $x^k$, i.e., the expectation is taken with respect to the randomness of $g^k$. First notice that
\begin{equation}\label{eq:unbiased_xx}
\E{g^k} = \nabla f(x^k).
\end{equation}

For any differentiable function $h$ let $D_h(x,y)$ to be Bregman distance with kernel $h$, i.e.,  $D_h(x,y)\eqdef h(x)-h(y) - \langle \nabla h(y),x-y\rangle$. Since
\begin{equation}\label{eq:prox_opt_xx}
x^* = \proxgjs(x^* - \alpha \nabla f(x^*)),
\end{equation}
and since the prox operator is non-expansive, we have
\begin{eqnarray}
\E{\norm{x^{k+1} -x^*}^2 } &\overset{ \eqref{eq:prox_opt_xx}}{=} &
 \E{\norm{\proxgjs(x^k-\alpha g^k) - \proxgjs(x^*-\alpha \nabla f(x^*))  }^2}  
 \notag \\
 &\stackrel{\eqref{eq:prox_cont}+\eqref{eq:q_identity}}{\leq} &
 \E{\norm{x^k-  x^* -\alpha \Popt(   g^{k} - \nabla f(x^*) ) }^2}  
 \notag\\
& \overset{\eqref{eq:unbiased_xx}}{=}& 
 \norm{x^k  -x^*}^2 -2\alpha \<  \nabla f(x^k)- \nabla f(x^*) , x^k  -x^*> \notag \\ 
 && \qquad  + \alpha^2\E{\norm{ g^{k}- \nabla f(x^*) }^2_{\Popt}}
 \notag  \\ 
&\leq & 
   (1-\alpha\sigma')\norm{x^k  -x^*}^2 +\alpha^2\E{\norm{  g^{k} -\nabla f(x^*) }^2_{\Popt}}  \notag \\
   && \qquad -2\alpha D_f(x^k,x^*).
 \label{eq:convstepsub1XX_v2}
\end{eqnarray}

Since $f(x)=\frac{1}{n}\sum_{i=1}^n f_i(x)$, in view of \eqref{eq:smooth} we have
\begin{eqnarray}
D_{f}(x^k,x^*) \quad = \quad  \frac{1}{n}\sum_{i=1}^n D_{f_i}(x^k,x^*) & \overset{\eqref{eq:smooth}}{\geq} &\frac{1}{2n} \sum_{i=1}^n \norm{\nabla f_i(x^k) -\nabla f_i(x^*)  }^2_{{\mM_i^{\dagger}} }  \notag\\
&=&  
 \frac{1}{2n}  \left\|{\cM^\dagger}^{\frac12}\left(\mG(x^k) -\mG(x^*) \right) \right\|^2. \label{eq:nb98gd8ff_v2}
\end{eqnarray}

By combining \eqref{eq:convstepsub1XX_v2} and \eqref{eq:nb98gd8ff_v2}, we get
\begin{eqnarray*}
\E{\norm{x^{k+1} -x^*}^2 } & \leq &
   (1-\alpha\sigma')\norm{x^k  -x^*}^2 +\alpha^2\E{\norm{g^{k} -\nabla f(x^*) }_{\Popt}^2}  \\
   && \qquad -\frac{\alpha}{n} \left\|{\cM^\dagger}^{\frac12}( \mG(x^k) -\mG(x^*)) \right\|^2.
\end{eqnarray*}

Next, applying Lemma~\ref{lem:g_lemma} with $\mQ = \Popt$ leads to the estimate
\begin{eqnarray}
\E{\norm{x^{k+1} -x^*}^2 } &\leq &
   (1-\alpha\sigma')\norm{x^k  -x^*}^2 -\frac{\alpha}{n} \norm{ {\cM^\dagger}^{\frac12} \left(\mG(x^k) -\mG(x^*) \right) }^2 \notag 
   \\
   && \qquad  +  \frac{2\alpha^2}{n^2} \E{ \norm{ \cU  \left( \mG(x^k) - \mG(x^*) \right)e }_{\Popt}^2  } \notag \\
   && \qquad + \frac{2\alpha^2}{n^2}  \E{\left\|\cU \left( \mJ^k - \mG(x^*) \right)e \right\|_{\Popt}^2}   \label{eq:48u34719841234_v2} .
\end{eqnarray}

Since, by assumption,  both $\cB$ and ${\cM^\dagger}^{\frac12}$ commute with $\cS$, so does their composition $\cA \eqdef \cB {\cM^\dagger}^{\frac12}$. Applying Lemma~\ref{lem:nb98gd8fdx}, we get
 \begin{eqnarray}\label{eq:J_jac_bound}\nonumber
\E{ \NORMG{\cB {\cM^\dagger}^{\frac12} \left(\mJ^{k+1}-\mG(x^*)  \right) }}  &=&  \NORMG{ (\cI - \E{\cS})^{\frac12}  \cB {\cM^\dagger}^{\frac12} \left(\mJ^k-\mG(x^*) \right) }  \\
&& +  \NORMG{\E{\cS}^{\frac12}  \cB {\cM^\dagger}^{\frac12} \left(\mG(x^k)-\mG(x^*) \right) } .
\end{eqnarray}

Adding $\alpha$-multiple of~\eqref{eq:J_jac_bound} for $\cC = {\cM^\dagger}^{\frac12}$ to~\eqref{eq:48u34719841234_v2} yields
\begin{eqnarray*}
&&\E{\norm{x^{k+1} -x^*}^2 } + \alpha\E{ \NORMG{\cB\left({\cM^\dagger}^{\frac12}(\mJ^{k+1}-\mG(x^*) )\right)}}
\\
&& \qquad  \leq 
   (1-\alpha\sigma')\norm{x^k  -x^*}^2 + \frac{2\alpha^2}{n^2} \E{  \norm{ \cU \left(\mG(x^k) - \mG(x^*) \right)e}_{\Popt}^2  } +   \\
   && \qquad \qquad  
 \frac{2\alpha^2}{n^2}  \E{\|\cU(\mJ^k - \mG(x^*))e \|_{\Popt}^2} +   \alpha\NORMG{ (\cI - \E{\cS})^{\frac12} \left(\cB\left({\cM^\dagger}^{\frac12}(\mJ^k-\mG(x^*))\right)\right)} 
 \\
    && \qquad \qquad   + \alpha \NORMG{\E{\cS}^{\frac12}  \left(\cB \left({\cM^\dagger}^{\frac12}(\mG(x^k)-\mG(x^*))\right)\right)}-\frac{\alpha}{n} \left\| {\cM^\dagger}^{\frac12}(\mG(x^k) -\mG(x^*)) \right\|^2
   \\
&& \qquad  \stackrel{\eqref{eq:small_step_v2}}{\leq} 
(1-\alpha\sigma')\norm{x^k  -x^*}^2 + (1-\alpha\sigma')\alpha \E{\NORMG{\cB\left({\cM^\dagger}^{\frac12} \left(\mJ^{k}-\mG(x^*) \right)\right)}} \\
&& \qquad \qquad  
+\frac{2\alpha^2}{n^2} \E{  \| \cU(\mG(x^k) - \mG(x^*))e\|_{\Popt}^2  } +  \alpha \NORMG{\E{\cS}^{\frac12}  \left(\cB \left({\cM^\dagger}^{\frac12}(\mG(x^k)-\mG(x^*))\right)\right)}  \\
    && \qquad \qquad  
 -\frac{\alpha}{n} \left\| {\cM^\dagger}^{\frac12}(\mG(x^k) -\mG(x^*)) \right\|^2
   \\
&& \qquad  \stackrel{\eqref{eq:small_step2_v2}}{\leq} 
(1-\alpha\sigma') \left( \norm{x^k  -x^*}^2 +\alpha\E{ \NORMG{\cB\left({\cM^\dagger}^{\frac12}(\mJ^{k}-\mG(x^*) )\right)}} \right).
\end{eqnarray*}
Above, we have used~\eqref{eq:small_step_v2} with $\mX=\mJ^k-\mG(x^*)$ and~\eqref{eq:small_step2_v2} with $\mX= \mG(x^k) -\mG(x^*)$.

 \subsection{Proof of Theorem~\ref{thm:sega_as}}
First, due to our choice of $\cS$ we have

\begin{eqnarray*}
\E{\cS (x)} &=& \diag(p) x
\end{eqnarray*}
and at the same time $\cS$ and ${\cM^\dagger}^{\frac12}$ commute. Next,~\eqref{eq:ESO_sega_good} implies
\begin{equation*}
\E{\left\|\cU(\mM^{\frac12} x)\right\|^{2}_{\Popt}} = \| x \|^2_{\mM^{\frac12} \E{\sum_{i\in S}\pLi^{-1}\eLi \eLi^\top \Popt \sum_{i\in S}\pLi^{-1} \eLi \eLi^\top }\mM^{\frac12} }   \leq  \| x \|^2_{p^{-1}\circ w}.
\end{equation*}

In order to satisfy~\eqref{eq:small_step_v2} and \eqref{eq:small_step2_v2} it remains to have (we substituted $y = {\mM^\dagger}^{\frac12}x$):
\begin{equation} \label{eq:linear_sega_1}
2\alpha\|y \|^2_{p^{-1}\circ w}+ 
\left\|\left(\cI-\E{\cS}\right)^{\frac{1}{2}}\cB(y)\right\|^{2} \leq(1-\alpha \sigma)\|\cB(y)\|^{2},
\end{equation}

\begin{equation} \label{eq:linear_sega}
2\alpha\|y \|^2_{p^{-1}\circ w} + 
\left\|\left(\E{\cS}\right)^{\frac{1}{2}}\cB(y)\right\|^{2} \leq \|y\|^2.
\end{equation}

Let us consider $\cB$ to be the operator corresponding to the left multiplication with matrix $\diag(b)$. Thus for satisfy~\eqref{eq:small_step_v2} it suffices to have for all $i \in [d]$:
\[
2\alpha m_i \pLi^{-1} + b_i^2(1-\pLi) \leq b_i^2 (1-\alpha \sigma) \qquad \Rightarrow \qquad
2\alpha m_i \pLi^{-1} + b_i^2\alpha \sigma \leq b_i^2\pLi .
\]
For~\eqref{eq:small_step2_v2} it suffices to have for all $i\in [d]$
\[
2\alpha m_i \pLi^{-1} + b_i^2 \pLi \leq 1 .
\]
It remains to notice that choice $b_i^2 = \frac{1}{2\pLi}$ and $\alpha  = \min_i \frac{\pLi}{4m_i+ \sigma}$ is valid.

\section{Experiments: The choice of the objective\label{sec:exp_choice}}
In all experiments from section~\ref{sec:experiments}, we have chosen $f(x) = \frac12 x^\top \mM x - b^\top x$ where $x\in \R^{1000}$, while $\psi$ is an indicator function of the unit ball intersected with $\Range{\Popt}$. First, matrix $\mM$ was chosen according to Table~\ref{tbl:M}. Next, vector $b$ was chosen as follows: first we generate $\tilde{x}\in \R^d$ with independent normal entries, then compute $\tilde{b} = \mM^{-1} \tilde{x}$ and set $b = \frac{3}{2 \| \tilde{b} \|}  \tilde{b}$. Lastly, for Figure~\ref{fig:variable}, the projection matrix $\Popt$ of rank $r$ was chosen as a block diagonal matrix with $r$ blocks, each of them being the matrix of ones multiplied by $\frac{r}{d}$.

\begin{table}[!h]
\caption{Choice of $\mM$. $\Odd$ is set of all odd positive integers smaller than $d+1$, while matrix $\mU$ was set as random orthonormal matrix (generated by QR decomposition from a matrix with independent standard normal entries).}
\label{tbl:M}
\begin{center}
\begin{tabular}{|c|c|c|c|}
\hline
Type & $\mM$  & Figure~\ref{fig:identity}: $L$ &  Figure~\ref{fig:variable}: $L$ \\
\hline
\hline
1 & $\mU \left( \mI + \mI_{:, \Odd}  \diag\left( ((L-1)^{\frac{1}{500}})^{(1:500)} \right)  \mI_{\Odd, :} \right) \mU^\top $  & 100 & 1000 \\
\hline
2 & $\mU \left( \mI +\sum_{i=1}^{100} (L-1) e_ie_i^\top \right) \mU^\top $ & 100 & 1000 \\
\hline
3&$\mU \left( \kappa \mI -\sum_{i=1}^{100} (L-1) e_ie_i^\top \right) \mU^\top $ &100 & 1000 \\
\hline
4 & $\left( \mI + \frac{L}{500}\mI_{:, \Odd}  \diag\left(1:500 \right)  \mI_{\Odd, :} \right) $& 100 & 1000 \\
\hline
\end{tabular}
\end{center}
\end{table}

\end{document}

%% file: all_commands.tex
\pdfoutput=1

\usepackage{amsfonts}
\usepackage{amsmath}

\usepackage{mdframed} 
\usepackage{thmtools}
\usepackage{mathtools}

\declaretheorem[within=section]{definition}
\declaretheorem[sibling=definition]{theorem}
\declaretheorem[sibling=definition]{proposition}
\declaretheorem[within=section]{assumption}
\declaretheorem[sibling=definition]{corollary}
\declaretheorem[sibling=definition]{remark}
\declaretheorem[sibling=definition]{example}
\declaretheorem[sibling=definition]{lemma}

\usepackage{colordvi,epsfig}
\usepackage[table]{xcolor}   
\usepackage{verbatim}

\usepackage{longtable}

\usepackage{verbatim,tikz}

\DeclareMathOperator*{\argmin}{argmin}

\graphicspath{{./figures/}}  

\newcommand{\ignore}[1]{}
\newcommand{\R}{\mathbb{R}}

\newcommand{\proxop}{\mathop{\mathrm{prox}}\nolimits}
\newcommand{\prox}{\proxop_{\alpha \psi}}
\newcommand{\proxgjs}{\proxop_{\alpha \varphi}}

\newcommand{\eqdef}{\coloneqq} 

\newcommand{\diag}{\mD} 
\newcommand{\blockdiag}{\text{BlockDiag}}

\newcommand{\cA}{{\cal A}}
\newcommand{\cB}{{\cal B}}
\newcommand{\cC}{{\cal C}}

\newcommand{\cI}{{\cal I}}

\newcommand{\cM}{{\cal M}}

\newcommand{\cO}{{\cal O}}

\newcommand{\cS}{{\cal S}}

\newcommand{\cU}{{\cal U}}

\newcommand{\mA}{{\bf A}}

\newcommand{\mD}{{\bf D}}

\newcommand{\mG}{{\bf G}}

\newcommand{\mI}{{\bf I}}
\newcommand{\mJ}{{\bf J}}

\newcommand{\mM}{{\bf M}}

\newcommand{\mQ}{{\bf Q}}

\newcommand{\mU}{{\bf U}}

\newcommand{\mW}{{\bf W}}
\newcommand{\mX}{{\bf X}}
\newcommand{\mY}{{\bf Y}}
\newcommand{\mZ}{{\bf Z}}

\newcommand{\ones}{e}

\newcommand{\E}[1]{\mathbb{E}\left[#1\right] } 
 
\newcommand{\Prob}[1]{\mathbb{P}\left(#1\right) }

\newcommand{\norm}[1]{\left \| #1 \right\|}

\providecommand{\Null}[1]{{\rm Null}\left( #1\right)}

\providecommand{\Range}[1]{{\rm Range}\left( #1\right)}

\usepackage[colorinlistoftodos,bordercolor=orange,backgroundcolor=orange!20,linecolor=orange,textsize=scriptsize]{todonotes}

\providecommand{\Vect}[1]{\mbox{Vec}\left( #1\right)}

\def\<#1,#2>{\left\langle #1,#2\right\rangle}

\newcommand{\NORMG}[1]{\left \| #1\right  \|^2}

\newcommand{\compactify}{} 

\newcommand{\eR}{{\color{blue}e}} %

\newcommand{\probx}{{\color{black} \rho}}
\newcommand{\Popt}{{\color{black}\mW}}

\newcommand{\pLi}{{\color{black}p_i}} 
\newcommand{\eLi}{{\color{black}e_i}} 

\newcommand{\Ind}[1]{I_{#1} } 
\newcommand{\Lift}[1]{Q \left(#1\right) }

\newcommand{\pp}{{\color{red}p}} 
\newcommand{\vv}{{\color{red}v}}
\newcommand{\ff}{{\color{red} f}} 

\newcommand{\LL}{{\color{red} L}} 
\newcommand{\mmu}{{\color{red} \mu}} 
\newcommand{\dd}{{\color{red} d}} 
\newcommand{\sS}{{\color{red} S}} 
\newcommand{\ppp}{{\pp}} 
\newcommand{\ppsi}{{\color{red} \psi}} 
\newcommand{\aalpha}{{\color{red} \alpha}} 
\newcommand{\eeta}{{\color{red} \eta}} 
\newcommand{\ggamma}{{\color{red} \gamma}} 

\newcommand{\xx}{{\color{red} x}} 
\newcommand{\yy}{{\color{red} y}} 
\newcommand{\zz}{{\color{red} z}} 

\newcommand{\ww}{{\color{red} w}} 
\newcommand{\ggg}{{\color{red} g}} 
\newcommand{\mmM}{{\color{red}  {\bf M}}}
\newcommand{\ee}{{\color{red}  e}}
\newcommand{\PP}{{\color{red}  P}}

\newcommand{\ccL}{{\cal L}}
\newcommand{\Odd}{{O_{dd} }}

\newcommand{\cL}{{\color{blue}\ccL'}}
\newcommand{\cLL}{{\color{red} \ccL}} 

\newcommand{\BD}{{\bf D_B}}